\documentclass[a4paper, 12pt]{article}
\usepackage[utf8]{inputenc}
\usepackage{a4}
\usepackage[T1]{fontenc}
\usepackage[english]{babel}
\usepackage{caption}
\usepackage{amssymb,amsmath,amsthm}
\usepackage{enumerate}
\usepackage{graphicx}
\usepackage{epsfig}
\usepackage{comment}
\usepackage{xspace}
\usepackage{float}
\usepackage{multirow}
\usepackage{extarrows}
\usepackage{mathtools}
\usepackage{array}
\usepackage{color}
\usepackage{nomencl}
\usepackage{nicefrac}
\usepackage{cite}
\usepackage{hyperref}
\makenomenclature
\newcommand{\N}{\ensuremath{\mathbb{N}}\xspace}

\newcommand{\R}{\ensuremath{\mathbb{R}}\xspace}

\newcommand{\eps}{\epsilon}

\renewcommand{\epsilon}{\varepsilon}

\newcommand{\dt}{\, dt}
\newcommand{\ud}{\ensuremath{u_{\delta}}\xspace}
\newcommand{\ul}{u_\lambda}
\newcommand{\dul}{\dot{u}_\lambda}
\newcommand\blfootnote[1]{%
  \begingroup
  \renewcommand\thefootnote{}\footnote{#1}%
  \addtocounter{footnote}{-1}%
  \endgroup
}

\makeatletter
\newcommand{\thickhline}{%
    \noalign {\ifnum 0=`}\fi \hrule height 2pt
    \futurelet \reserved@a \@xhline
}
\makeatother
\begin{document}

\numberwithin{equation}{section}

\newtheoremstyle{break}{15pt}{15pt}{\itshape}{}{\bfseries}{}{\newline}{}
\theoremstyle{break}
\newtheorem*{Satz*}{Theorem}
\newtheorem*{Rem*}{Remark}
\newtheorem*{Lem*}{Lemma}
\newtheorem{Satz}{Theorem}[section]
\newtheorem{Rem}{Remark}[section]
\newtheorem{Lem}{Lemma}[section]
\newtheorem{Prop}{Proposition}[section]
\newtheorem{Cor}{Corollary}[section]
\theoremstyle{definition}
\newtheorem{Def}[Satz]{Definition}
\theoremstyle{plain}
\newtheorem{exmp}{Example}[section]
\parindent2ex
\newenvironment{rightcases}
  {\left.\begin{aligned}}
  {\end{aligned}\right\rbrace}

\begin{center}{\Large \bf Signal recovery via TV-type energies}
\end{center}

\begin{center}
M. Fuchs, J. M\"uller, C. Tietz
\end{center}
\textit{Dedicated to the memory of Stefan Hildebrandt}


\begin{abstract}
We consider one-dimensional variants of the classical first order total variation denoising model introduced by Rudin, Osher and Fatemi. This study is based on our previous work on various denoising and inpainting problems in image analysis, where we applied variational methods in arbitrary dimensions. More than being just a special case, the one-dimensional setting allows us to study regularity properties of minimizers by more subtle methods that do not have correspondences in higher dimensions. In particular, we obtain quite strong regularity results for  a class of data functions that contains many of the standard examples from signal processing such as rectangle- or triangle signals as a special case. An analysis of the related Euler-Lagrange equation, which turns out to be a second order two-point boundary value problem with Neumann conditions, by ODE methods completes the picture of our investigation.
\end{abstract}

\begin{section}{Introduction}
\blfootnote{\textbf{Acknowledgement}:
The authors thank Michael Bildhauer for many stimulating discussions.}
Since the publication of the  seminal paper \cite{ROF} of Rudin, Osher and Fatemi in 1992, total variation based denoising and inpainting methods have proved to be very effective when dealing with two- or higher dimensional noisy data such as digital images, which nowadays has become their main field of application. However, their one-dimensional counterparts in signal processing seem to find usage as well, mainly in connection with the recovery of piecewise constant data as it is frequently encountered in many practical sciences such as geophysics or biophysics (cf. \cite{LJ} and the introduction of \cite{BPS}), whereas in e.g. \cite{FIMT} TV-models have been applied to the filtering of gravitational wave signals. Apart from the variety of possible applications, our interest in the one-dimensional case primarily stems from our previous work on TV-based variational problems in image analysis. In \cite{BF1, BF2, BF3, BF5, BF6,BFT,FT}, variants of the classical TV-functional have been studied in any dimension replacing the regularization term by a convex functional of linear growth, which approximates the TV-seminorm and in addition has suitable ellipticity properties that make the considered models more feasible to analytical techniques. When trying to improve our results in the one-dimensional setting, we found ourselves surprised that first, this is not a consequence of completely elementary arguments and second, there are certain features of the corresponding solutions that do not seem to have analogues in arbitrary dimensions. In this context, we would also like to mention the works \cite{BKV}, \cite{CS} and \cite{BP} where similar considerations have been applied to study  the classical TV-model as well as its generalizations towards functionals that involve higher derivatives in one dimension.

We proceed with a precise formulation of our setting and results.\\
Let $f\in L^\infty(0,1)$ be a given function which represents an observed signal (possibly corrupted by an additive Gaussian noise). We will always assume $0\leq f\leq 1$ a.e. on $(0,1)$. For a given density function $F:\R\rightarrow[0,\infty)$ of linear growth we consider the following minimization problem: 

\begin{align}\label{model}
 J[u]:=\intop_{0}^1 F(\dot{u})\,dt+\frac{\lambda}{2}\intop_{0}^1(u-f)^2\,dt\rightarrow\min.
\end{align}
Here, $\intop dt$ is Lebesgue's integral in one dimension, $\dot{u}:=\frac{d}{dt}u$ denotes the (weak) derivative of a function $u:(0,1)\rightarrow\R$ and $\lambda>0$ is a regularization parameter which controls the balance between the smoothing and the data-fitting effect resulting from the minimization of the first and the second integral respectively. We impose the following mild conditions on our energy density $F$:
\begin{align}
 &F\in C^2(\R),\;F(-p)=F(p), F(0)=0,\label{F1} \tag{F1} \\ 
 &|F'(p)|\leq\nu_1,\label{F2}\tag{F2} \\ 
 &F(p)\geq \nu_2|p|-\nu_3\text{ and } \label{F3}\tag{F3} \\ 
 &F''(p)>0 \label{F4}\tag{F4}
\end{align}
for all $p\in\R$ and for some constants $\nu_1,\nu_2>0$, $\nu_3\in\R$. Note that from (\ref{F1}) and (\ref{F2}) it follows $F(p)\leq \nu_1|p|$ for all $p\in\R$. Moreover it should be obvious that the condition $F(0)=0$ is just imposed for notational simplicity. Examples of a reasonable choice of $F$ are given by the regularized TV-density, $F_\eps(p):=\sqrt{\eps^2+p^2}-\eps$ for some $\eps>0$ or $F(p):=\Phi_\mu(|p|)$, where $\Phi_\mu$ denotes the standard example of a so called $\mu$-elliptic density, i.e. for a given ellipticity parameter $\mu>1$ we consider
\[
 \Phi_{\mu}(r):=\int\limits_{0}^{r}{\int\limits_{0}^{s}{(1+t)^{-\mu}\dt}\,ds},\quad r\geq0,
\]
and observe the formulas
\begin{align}\label{Phimu}
 \left\{\begin{aligned}
 &\Phi_\mu(r)=\frac{1}{\mu-1}r+\frac{1}{\mu-1}\frac{1}{\mu-2}(r+1)^{-\mu+2}-\frac{1}{\mu-1}\frac{1}{\mu-2},\;\mu\neq 2,\\
 &\Phi_2(r)=r-\ln(1+r),\;r\geq 0.
\end{aligned}\right.
\end{align}
It is easily confirmed that $F(p):=\Phi_\mu(|p|)$, $p\in\R$, satisfies the condition of $\mu$-ellipticity
\begin{align}\label{muell}\tag{F5}
 F''(p)\geq \frac{c_1}{(1+|p|)^\mu}
\end{align}
for a constant $c_1>0$  as well as (\ref{F1})-(\ref{F4}). We remark that we have the validity of 
\[
 \lim\limits_{\mu\rightarrow\infty}{(\mu-1)\Phi_{\mu}(|p|)}=|p|,
\]
which underlines that $\Phi_{\mu}(|p|)$ is a good candidate for approximating the TV-density (see, e.g. \cite{BF1}, \cite{BF2}, \cite{BF3} or \cite{FT}). We further define the positive number
\begin{align}\label{lambdainf}
 \lambda_\infty=\lambda_{\infty}(F):=\lim_{p\rightarrow\infty}F'(p).
\end{align}
This value will turn out to be sort of a natural threshold in the investigation of the regularity properties of minimizers of problem (\ref{model}). 
\begin{exmp}\label{exmp1}
 For $F_\eps$ it is immediate that $\lambda_\infty(F_\eps)=1$ independently of $\eps$, whereas for $F=\Phi_\mu$ we have
\[
\lambda_\infty(\Phi_\mu)=\frac{1}{\mu-1}.
\]
\end{exmp}

Before giving a résumé of our results concerning problem (\ref{model}), we have to add some comments on functions and related spaces. For a general overview on one-dimensional variational problems and a synopsis of the related function spaces, we refer to \cite{BGH}. For $1\leq p\leq \infty$ and $m\in\N$ we denote the standard Sobolev space on the interval $(0,1)$ of (locally) $m$-times weakly in $L^p(0,1)$ differentiable functions by $W^{m,p}_{(\mathrm{loc})}(0,1)$ equipped with the norm $\|\cdot\|_{m,p}$. For a more detailed analysis of these spaces we refer to classical textbooks on this subject such as e.g. \cite{Ad}. We will frequently make use of the identification $ W^{m,\infty}(0,1)=C^{m-1,1}\big([0,1]\big)$, where for $0<\alpha\leq 1$ $C^{m,\alpha}(0,1)$ as usual denotes the space of $m$-times differentiable functions with locally Hölder continuous derivatives on $(0,1)$ and $C^{m,\alpha}\big([0,1]\big)$ has an obvious meaning. In the case $\alpha=1$ and $m=0$, $C^{0,1}\big([0,1]\big)=W^{1,\infty}(0,1)=:\mathrm{Lip}(0,1)$ is the space of  Lipschitz-continuous functions where our notion makes implicit use of the fact that these functions posses a Lipschitz continuous extension to the boundary. We further would like to remark that some authors prefer to write $AC(0,1)$ in place of $W^{1,1}(0,1)$ referring to the more classical notion of ``absolutely continuous'' functions forming a proper subspace of $C^{0}\big([0,1]\big)$, see e.g. \cite{BGH}, Chapter 2. Finally, $BV(0,1)$ denotes the space of functions of bounded variation on $(0,1)$, i.e. the set of all functions $u\in L^1(0,1)$ whose distributional derivative $Du$ can be represented by a signed Radon measure of finite total mass $\int\limits_{0}^{1}{|Du|}$. For more information concerning these spaces the reader is referred to the monographs \cite{Giu} and \cite{AFP}.\\
Due to \cite{AFP}, Theorem 3.28, p. 136 (see also section 2.3 on p. 90 in \cite{BGH}), there is always a ``good`` representative of a $BV$-function $u$ which is continuous up to a countable set of jump points $\{x_k\}\subset(0,1)$, $k\in\N$, i.e. in particular,  the left- and the right limit exist at all points. In what follows, we will tacitly identify any $BV$-function with this particular representative. We further note that the classical derivative of this good representative, which we denote by $\dot{u}$, exists at almost all points (see \cite{AFP}, Theorem 3.28, p.136 once again). The measure $Du$  can then be decomposed into the following sum
\begin{align}
\label{Du} Du=\overset{\substack{\mbox{$=:D^au$}\\ \vspace{-0.15cm} \\ \overbrace{\hspace{1.2cm}}}}{\dot{u}\mathcal{L}^1}+\overset{\mbox{$=:D^su$}}{\overbrace{\sum_{k\in\N}h(x_k)\delta_{x_k}+D^cu}},
\end{align}
and it holds (compare \cite{AFP}, Corollary 3.33)
\[
 |Du|(0,1)=\int\limits_{0}^{1}{|\dot{u}|dt}+\sum\limits_{k\in\N}{|h(x_k)|}+|D^cu|(0,1).
\]
Here, $h(x_k):=\lim\limits_{x\downarrow x_k}u(x)-\lim\limits_{x\uparrow x_k}u(x)$ denotes the "jump-height"  and $\delta_{x_k}$ is Dirac's measure  at $x_k$. The sum $\sum_{k\in\N}h(x_k)\delta_{x_k}$ is  named the \textit{jump part} $D^ju$ of $Du$ which, together with the so called \textit{Cantor part} $D^cu$ forms the \textit{singular part} $D^su$. Furthermore, $\dot{u}\mathcal{L}^1$ is the \textit{absolutely continuous part} $D^au$ w.r.t. the measure $\mathcal{L}^1$ and $D^au+D^su$ is the Lebesgue decomposition of $Du$. 

Coming back to the subject of our investigation we put problem (\ref{model}) in a more precise form, i.e. we consider the minimization problem:
\begin{align}
\label{J} J[w]:=\int\limits_{0}^{1}{F(\dot{w})\dt}+ \frac{\lambda}{2}\int\limits_{0}^{1}{(w-f)^2\dt}\rightarrow\text{min in}\,\,W^{1,1}(0,1)
\end{align}
for a density $F$ satisfying (\ref{F1})-(\ref{F4}), in particular $F$ is of linear growth. Hence, the Sobolev space $W^{1,1}(0,1)$ is the natural domain of  $J$ . However, due to the non-reflexivity of this space we can in general not expect to find a solution. Following ideas in \cite{BF1}, we therefore pass to a relaxed version $K$ of the above functional which is defined for $w\in BV(0,1)$ and takes a particularly simple form in our one-dimensional setting, which means that we replace (\ref{J}) by the problem
\begin{align}
\label{K}\begin{split} 
K[w]:=\intop_0^1 F(\dot{w})\dt+\lambda_\infty|D^sw|(0,1)+&\frac{\lambda}{2}\intop_{0}^1(w-f)^2\dt\\
&\rightarrow\text{min in}\,\,BV(0,1).
\end{split}
\end{align}
We would like to note, that the above formula coincides with the usual notion of relaxation in $BV$ (cf. \cite{AFP}, Theorem 5.47 on p. 304) since under the assumptions imposed on $F$ the recession function $F^\infty(p):=\lim\limits_{t\rightarrow\infty}F(tp)/t$  simplifies to $F^\infty(p)= \lambda_\infty|p|$.

From the point of view of regularity, $BV$-minimizers (i.e. solutions of problem (\ref{K})) are not very popular. However, it turns out that if we (strongly) restrict the size of the free parameter $\lambda$ it is possible to establish existence of a unique $J$-minimizer $u$ in the space $W^{1,1}(0,1)$. Part a) of the following theorem is concerned with this issue whereas in part b) we show that the minimizer of the relaxed variant (\ref{K}) in the space $BV(0,1)$ is exactly the solution $u$ from part a). Part c) is devoted to the regularity behaviour of the $J$-minimizer $u$. Here we can prove optimal regularity, which in this context means that $u$ is of class $C^{1,1}$ on the interval $[0,1]$. Furthermore, it turns out that $u$ solves a Neumann-type two-point boundary value problem. Precisely we have
\begin{Satz}[full regularity for small values of $\lambda$]\label{Thm1.1}
Suppose that $0\leq f\leq1$ a.e. on $[0,1]$ and that the density $F$ satisfies (\ref{F1})--(\ref{F4}). We further assume that the parameter $\lambda$ satisfies
\begin{align}
 \label{criticallambda} \lambda<\lambda_{\infty}(F)
\end{align}
with $\lambda_{\infty}(F)$ defined in (\ref{lambdainf}). Then it holds:
\begin{enumerate}[a)]
\item  Problem (\ref{J}) admits a unique solution $u\in W^{1,1}(0,1)=AC(0,1)$ and this solution satisfies $0\leq u(x)\leq 1$ for all $x\in[0,1]$.
\item The relaxation \grqq $K\rightarrow\text{min in}\,\, BV(0,1)$\grqq\,\,has just one solution which coincides with $u$ from part a).
 \item The minimizer $u$ belongs to the class $W^{2,\infty}(0,1)=C^{1,1}\big([0,1]\big)$ and solves the following Neumann-type boundary value problem 
\begin{align}\tag{BVP}\label{ode}
  \left\{\begin{aligned}
   &\ddot{u}=\lambda \frac{u-f}{F''(\dot{u})}\text{ a.e. on }(0,1),\\
   &\dot{u}(0)=\dot{u}(1)=0.
  \end{aligned}\right.
\end{align}
\end{enumerate}
\end{Satz}

\begin{Rem}\label{optimallambda}
The bound (\ref{criticallambda}) on the parameter $\lambda$ occurs for technical reasons since it allows us to prove a general statement on the solvability of problem (\ref{J}). In practice, this threshold strongly depends on the data function $f$ as well and as numerical experiments suggest, often exceeds $\lambda_\infty$. In Theorems \ref{Thm1.4} and \ref{Thm1.3} we will determine better estimates for $\lambda$ under which we can expect solvability of (\ref{J}) whereas Theorem \ref{Thm1.5} proves, that the statement of Theorem \ref{Thm1.1} is indeed only true for a restricted range of $\lambda$. 
\end{Rem}

Next we drop the bound (\ref{criticallambda}) and pass to the relaxed variational problem (\ref{K}).
\begin{Satz}[partial regularity for arbitrary values of $\lambda$]\label{Thm1.2}
Suppose that $0\leq f\leq1$ a.e. on $[0,1]$ and let the density $F$ satisfy (\ref{F1})--(\ref{F4}) as well as the additional requirement
\begin{align}
\label{F5}\tag{F6} F''(p)\leq c_2\frac{1}{1+|p|}
\end{align}
for all $p\in\R$, where $c_2>0$ is a constant. Moreover, let $\lambda>0$ denote any number. Then it holds:
\begin{enumerate}[a)] 
 \item Problem (\ref{K}) admits a unique solution $u\in BV(0,1)$ satisfying $0\leq u\leq 1$ a.e. 
\item There is an open subset $\mathrm{Reg}(u)$ of $(0,1)$ such that $u\in W^{2,\infty}_{\mathrm{loc}}(\mathrm{Reg}(u))$ and $\mathcal{L}^1\big((0,1)-\mathrm{Reg}(u)\big)=0$. We have 
\[
 \mathrm{Reg}(u):=\{s\in(0,1): s\,\,\text{is a Lebesgue point of}\,\,\dot{u}\},
\]
where $\dot{u}$ is defined in (\ref{Du}). Moreover, there are numbers $0<t_1\leq t_2<1$ such that $u\in C^{1,1}\big([0,t_1]\cup[t_2,1]\big)$.
\item If there is a subinterval $(a,b)\subset(0,1)$ such that $f\in W^{1,2}_{\mathrm{loc}}(a,b)$, then $u\in W^{1,1}(a,b)\cap W^{1,2}_{\mathrm{loc}}(a,b)$. In case $(a,b)=(0,1)$ we get that $u\in W^{1,1}(0,1)\cap W^{1,2}_{\mathrm{loc}}(0,1)$ is $J$-minimizing in $W^{1,1}(0,1)$.
\end{enumerate}
\end{Satz}

\begin{Rem}
\begin{enumerate}[(i)]
 \item Note that we need (\ref{F5}) only for proving part c). Parts a) and b) remain valid without (\ref{F5}).
 \item The requirement (\ref{F5}) is not as restrictive as it may appear at first sight. In particular, it is easy to confirm that for a given $\eps>0$ and $\mu>1$ our examples from the introduction, $F(p):=F_{\eps}(p)$ and $F(p):=\Phi_{\mu}(|p|)$ satisfy condition (\ref{F5}).
\end{enumerate}
\end{Rem}

Since signals in practice are usually modeled by more regular functions rather than just through measurable ones (we have e.g. rectangular- or 'sawtooth'-like signals in mind, which are differentiable outside a small exceptional set), it is reasonable to ask to what extend these properties are reproduced by the $K$-minimizer $u$. The next theorem shows how the results from Theorem \ref{Thm1.2} can be improved if we assume better data.

\begin{Satz}[regularity for special data]\label{Thm1.4}
Suppose that the density $F$ satisfies (\ref{F1})--(\ref{F4}), assume $0\leq f\leq 1$ a.e. on $[0,1]$ and let $u$ be the $K$-minimizer from Theorem \ref{Thm1.2}.
\begin{enumerate}[a)]
\item Let $t_0\in (0,1)$ be a point, where some representative of the data function $f$ is continuous. Then the good representative of $u$ introduced in front of (\ref{Du}) is continuous at $t_0$.
\item Assume that there is an interval $[a,b]\subset (0,1)$ such that $f\in \mathrm{Lip}(a,b)=W^{1,\infty}(a,b)$. Then we have $u\in C^2(a,b)$.
\item Suppose $f\in W^{1,1}(0,1)$ and define 
\begin{align}\label{omegainfty}
 \omega_\infty:=\lim_{p\rightarrow\infty}pF'(p)-F(p)\in (0,\infty].
\end{align}
Then, if $\lambda\left(\frac{1}{2}+\|\dot{f}\|_1\right)<\omega_\infty$, it follows $u\in C^{1,1}\big([0,1]\big)$.
\end{enumerate}
\end{Satz}

\begin{Cor} \label{Coro} If the data function $f$ is globally Lipschitz-continuous on $[0,1]$, then it follows $u\in C^2\big([0,1]\big)$.
\end{Cor}
\begin{proof}[Proof of Corollary \ref{Coro}]
Applying Theorem \ref{Thm1.4} b) with $a$ and $b$ arbitrarily close to $0$ and $1$, respectively, yields $u\in C^2(0,1)$. In particular, it is therefore immediate that $u$ satisfies the differential equation from Theorem \ref{Thm1.1} c)
\begin{align}\label{difff}
 \ddot{u}=\lambda\frac{u-f}{F''(\dot{u})}
\end{align}
everywhere on $(0,1)$. Due to Theorem \ref{Thm1.2} b) we have $u\in C^1\big([0,t_1]\cup[t_2,1]\big)$ and therefore $\dot{u}$ is uniformly continuous on $[0,1]$, which means that the right-hand side of equation (\ref{difff}) belongs to the space $C^0\big([0,1]\big)$. Thus $\ddot{u}$ exists even in $0$ and $1$ and is a continuous function on $[0,1]$.
\end{proof}

\begin{Rem}\label{remthm13}
\begin{enumerate}[ (i)]
 \item From part a) we infer that if $f$ is continuous on an interval $(a,b)\subset[0,1]$ , then also $u\in C^0(a,b)$. 
 \item We would like to remark that part b) in particular applies to piecewise affine data functions such as triangular or rectangular signals as shown in figure \ref{fig1} below:
 \begin{figure}[h!]
\centering
  \includegraphics[scale=0.9]{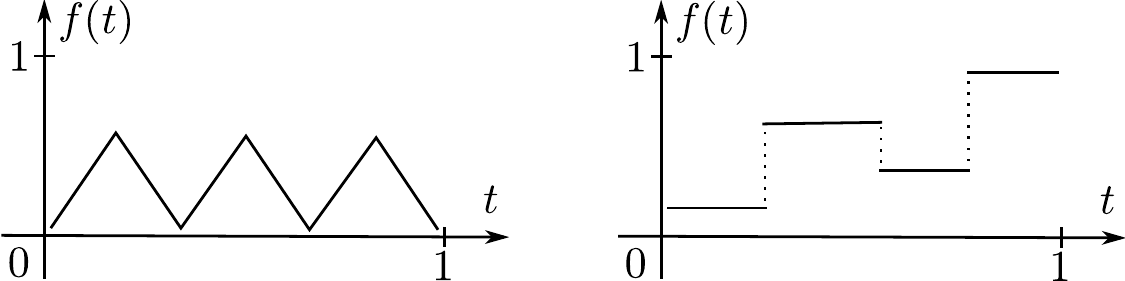}
  \caption{Examples of typical data functions}\label{fig1}
 \end{figure}

 We then obtain differentiability of the corresponding $K$-minimizers outside the set of jump points of the data. In particular, if the data are Lipschitz except for a countable set of jump-type discontinuities, then $K$ attains its minimum in the space $SBV(0,1)$ of special functions of bounded variation (see \cite{AFP}, chapter 4 for a definition). 
 \item The main feature of part c) of Theorem \ref{Thm1.4} is, that even though full $C^{1,1}$-regularity may fail to hold in general if the parameter $\lambda$ exceeds $\lambda_\infty$, it can still hold up to $2\lambda_\infty$ provided the oscillation of the data $f$ is sufficiently small.
If we take for example the regularized graph-length integrand as our density $F$, i.e. $F(p):=F_\eps(p)=\sqrt{\eps^2+p^2}-\eps$, it is easily verified that
\[
 \omega_{\infty}(F_\eps)=\eps.
\]
Consequently, we get full $C^{1,1}$-regularity for all parameters $\lambda$ up to the bound
\[
 \frac{\eps}{\frac{1}{2}+\|\dot{f}\|_1},
\]
which might be larger than $\lambda_\infty(F_\eps)=1$ provided we choose $\eps$ sufficiently large. If we take $F(p)=\Phi_{\mu}(|p|)$ it holds $\lambda_\infty=\frac{1}{\mu-1}$, whereas
\begin{align*}
\lim_{p\rightarrow\infty} p\Phi_\mu'(p)-\Phi_\mu(p)=
\left\{\begin{aligned}
&\frac{1}{\mu-1}\frac{1}{\mu-2},\quad&\mu> 2,\\
&\infty& 1<\mu\leq 2,
\end{aligned}\right.
\end{align*}
so that in particular $\omega_\infty$ is unbounded if we let $\mu$ approach $2$ from above. 
\end{enumerate}
\end{Rem}
Next we would like to demonstrate the sharpness of our previous regularity results, in particular we want to indicate that singular (i.e. discontinuous) minimizers can occur if we pass from Lipschitz signals $f$ studied in Theorem \ref{Thm1.4} (cf. also Corollary \ref{Coro}) to functions $f$ having jumps in some interior points of the interval $[0,1]$. To be precise, we let for $\mu>1$
\begin{align}\label{F=Phi}
F(p)=\Phi_\mu(|p|),\;p\in\R
\end{align}
with $\Phi_\mu$ as defined in (\ref{Phimu}) and recall that for this density we have (compare Example \ref{exmp1})
\begin{align}\label{lambdaphi}
\lambda_\infty=\frac{1}{\mu-1}.
\end{align}
Moreover we define
\begin{align}\label{specialdata}
 f:[0,1]\rightarrow[0,1],\,f(t)=\begin{cases}
       0,\;0\leq t\leq\frac{1}{2},\\
       1,\;\frac{1}{2}<t\leq 1.
      \end{cases}
\end{align}
Then it holds:
\begin{Satz}[existence of discontinuous minimizers]\label{Thm1.5}
Under the assumptions (\ref{F=Phi}) and (\ref{specialdata}) and with parameters $\lambda>0$, $\mu>1$ let $u\in BV(0,1) $ denote the unique solution of problem (\ref{K}) (being of class $C^2\big([0,1]-\{\frac{1}{2}\}\big)$ on account of Theorem \ref{Thm1.4} and an obvious modification of the proof of Corollary \ref{Coro}). Then, if we assume $\mu>2$ and if $\lambda$ satisfies
\begin{align}\label{lambdagrenz}
\lambda>\frac{8}{\mu-2}
\end{align}
it holds
\begin{align*}
\sup_{0\leq t<1/2}<\frac{1}{2}<\inf_{1/2<t\leq 1}u,
\end{align*}
which means that $u$ has a jump discontinuity at $t=1/2$.
\end{Satz}
\begin{Rem}\label{remark1}
From our previous works \cite{BF1}, \cite{BF2} and \cite{BFT} we see that for $\mu\in(1,2)$ and any $\lambda>0$ this phenomenon can not occur, i.e. the minimizer $u$ is a regular function. Thus the value $\mu=2$ separates regular from irregular behaviour of the solutions.
\end{Rem}
\begin{Rem}
Assume that $\lambda>0$ is fixed. Then it follows from (\ref{lambdagrenz}) that we can force the minimizer $u$ to create a jump point at $t=\nicefrac{1}{2}$ by choosing $\mu$ sufficiently large.
\end{Rem}
\begin{Rem}
On account of Theorem 1.1 the solution is regular provided that
\[
\lambda<\lambda_\infty\overset{(\ref{lambdaphi})}{=}\frac{1}{\mu-1}.
\]
On the other hand, inequality (\ref{lambdagrenz}) states
\[
\lambda>\frac{8}{\mu-2}=8\left(\frac{1}{\mu-1}+\frac{1}{(\mu-1)(\mu-2)}\right)=8(\lambda_\infty+\omega_\infty),
\]
which suggests that our solution is irregular, if $\lambda$ and $\mu$ are chosen in such a way that $\lambda>8(\lambda_\infty+\omega_\infty)$ (see Corollary \ref{Cor1.2} below).
\end{Rem}
With respect to Theorem \ref{Thm1.5} and Remark \ref{remark1} it remains to discuss the situation for the limit case $\mu=2$, which can be done in a very general form: it turns out that our arguments are valid for all $\mu$-elliptic densities $F$ with exponent $\mu\in(1,2]$ and for arbitrary measurable data $f$ leading to $C^{1,1}$-regularity of minimizers. It should be noted that this in particular implies the smoothness of minimizers in case $1<\mu<2$ without referring to the higher-dimensional results. Precisely we have
\begin{Satz}[regularity for $\mu$-elliptic densities for $1<\mu\leq 2$]\label{Thm1.6}
 Suppose $0\leq f\leq 1$ a.e. on $[0,1]$ and consider a density $F$ with (\ref{F1})-(\ref{muell}). Moreover, fix any number $\lambda>0$. Then, if
\begin{align}\label{muin}
 \mu\in (1,2]
\end{align}
holds, the unique solution $u\in BV(0,1)$ of problem (\ref{K}) is of class $C^{1,1}\big([0,1]\big)$.
\end{Satz}

 From the proofs of Theorem \ref{Thm1.5} and Theorem \ref{Thm1.6} we obtain the following slightly more general result on regular or irregular behaviour of minimizers avoiding the notion of $\mu$-ellipticity (\ref{muell}).
\begin{Cor}\label{Cor1.2}
Let $F$ satisfy (\ref{F1})-(\ref{F4}) and define $\omega_\infty$ as in (\ref{omegainfty}).
\begin{enumerate}[a)]
 \item In case $\omega_\infty=\infty$ any solution $u\in BV(0,1)$ is of class $C^{1,1}\big([0,1]\big)$ independent of the value of $\lambda$ and for arbitrary data $f\in L^\infty(0,1)$, $0\leq f\leq 1$ a.e.
 \item If $\omega_\infty<\infty$ and if (\ref{F5}) holds, then there is a critical value $\lambda_{\mathrm{crit}}$ of the parameter $\lambda$ such that the solution $u$ of (\ref{K}) with $f$ being defined in (\ref{specialdata}) is discontinuous (exactly at $t=\nicefrac{1}{2}$) provided we choose $\lambda>\lambda_{\mathrm{crit}}$. Moreover, it holds
\begin{align}
\max\{\lambda_\infty,8\omega_\infty\}\leq\lambda_{\mathrm{crit}}\leq 8(\lambda_\infty+\omega_\infty).
\end{align}
\end{enumerate}
\end{Cor}

\begin{Rem}
 Comparing part b) of the above corollary to parts a) and c) of Theorem \ref{Thm1.4}, we would like to emphasize that the occurence of discontinuous minimizers requires discontinuous data.
\end{Rem}

By part $c)$ of Theorem \ref{Thm1.1}, the minimization problem (\ref{model}) leads  to the second-order Neumann problem (\ref{ode}). Conversely, we could take this equation as our starting point and examine existence and regularity of solutions purely by methods from the theory of ordinary differential equations. In the articles \cite{Th1} and \cite{Th2}, Thompson has worked out an extensive theory for a large class of two-point boundary value problems with both continuous and measurable right-hand sides, which we could apply to our situation with the following result:
\begin{Satz}[regularity for $\mu$-elliptic densities and $\mu>1$ arbitrary]\label{Thm1.3}
Suppose $0\leq f(t)\leq 1$ a.e. on $[0,1]$ and let $F$ satisfy (\ref{F1})-(\ref{F3}) as well as (\ref{muell}).
If the parameter $\lambda$ satisfies
\[  
0<\lambda<\sup_{L>1}\frac{1}{c_1}\intop_{1}^L \frac{s\,ds}{(1+s)^\mu}=:\lambda_\mu, 
\]
where $c_1$ is as in (\ref{muell}), then there exists $v\in W^{2,1}(0,1)$, satisfying $0\leq v(t)\leq 1$ for  all $t\in[0,1]$ and which solves the Neumann problem (\ref{ode}) a.e. on $[0,1]$. Furthermore, this solution coincides with the unique $K$-minimizer $u$ from the space $BV(0,1)$.
\end{Satz}
\begin{Rem}\label{rem1.7}
\begin{enumerate}[(i)]
 \item The reader being familiar with the theory of lower and upper solutions will recognise the above bound $\lambda_\mu$ as a sort of "Nagumo-condition" (see, e.g. \cite{CH}), which guarantees a priori bounds on the first derivative of the solution $v$.
 \item If $f$ is continuous, the differential equation implies $v\in C^2\big([0,1]\big)$. 
 \item Using the example $F(p)=\Phi_\mu(|p|)$ we would like to demonstrate, how $\lambda_\mu$ might actually improve the bound for $\lambda$ stated in (\ref{criticallambda}) of Theorem \ref{Thm1.1}: obviously, the integral defining $\lambda_\mu$ diverges for $1<\mu\leq 2$ and is unbounded if $\mu$ approaches 2 from above. In combination with part (ii) of this remark, we consequently get full $C^{2}\big([0,1]\big)$-regularity for arbitrarily large values of the parameter $\lambda$ and continuous data $f$, if we let $\mu\downarrow 2$. In fact it holds $\lambda_\mu>\lambda_\infty(\Phi_\mu)=\frac{1}{\mu-1}$ up to $\mu\approx2.803$.
\end{enumerate}
 \end{Rem}
Since it is somewhat difficult to track the various regularity statements from Theorem \ref{Thm1.1} up to \ref{Thm1.3}, we have summarized our main results in form of a table. It shows the regularity of the $K$-minimizer $u$ dependent on the data $f$, the density $F$ and the bound on the parameter $\lambda$.

\begin{table}[h!]
\begin{tabular}{|c|c|c||c|c|}\hline
   Data $f$ & Density $F$ & Bound on $\lambda$ & Regularity of $u$ & Reference \\ \thickhline
   $L^\infty(0,1)$ & (\ref{F1})-(\ref{F4}) & $0<\lambda<\lambda_\infty$ & $C^{1,1}\big([0,1]\big)$ & Theorem \ref{Thm1.1} a) \\ \hline
   $L^\infty(0,1)$ & (\ref{F1})-(\ref{F4}) & $\lambda>0$ & $W^{2,\infty}_{\mathrm{loc}}\big(\mathrm{Reg}(u)\big)$ & Theorem \ref{Thm1.2} b) \\ \hline
   $W^{1,2}_{\mathrm{loc}}(a,b)$ & \begin{minipage}{2cm}(\ref{F1})-(\ref{F4}),\\ \hspace*{0.5cm}(\ref{F5})\end{minipage} & $\lambda>0$ & \begin{minipage}{2.2cm}$W^{1,1}(a,b)\ \\ \cap W^{1,2}_{\mathrm{loc}}(a,b)$\end{minipage} & Theorem \ref{Thm1.2} c) \\ [1.2ex] \hline  
   \begin{minipage}{2cm} continuous \\  \hspace*{0.4cm} at $t_0$\end{minipage} & (\ref{F1})-(\ref{F4}) & $\lambda>0$ & \begin{minipage}{2cm} continuous \\ \hspace*{0.4cm} at $t_0$\end{minipage} & Theorem \ref{Thm1.4} a) \\ [1ex] \hline 
   $W^{1,1}(0,1)$ & (\ref{F1})-(\ref{F4}) & \begin{minipage}{2cm}$\lambda(\frac{1}{2}+\|\dot{f}\|_1)$\\\hspace*{0.4cm} $<\omega_\infty$ \end{minipage}& $C^{1,1}\big([0,1]\big)$ & Theorem \ref{Thm1.4} c) \\ [1.2ex] \hline 
   $L^\infty(0,1)$ &  \begin{minipage}{2cm}(\ref{F1})-(\ref{muell})\\ \hspace*{0.1cm}$\mu\in (1,2]$\end{minipage} & $\lambda>0$ & $C^{1,1}\big([0,1]\big)$ & Theorem \ref{Thm1.6} \\ [1.2ex]\hline
   $L^\infty(0,1)$ & \begin{minipage}{2cm}(\ref{F1})-(\ref{F4})\\ \hspace*{0.1cm}$\omega_\infty=\infty$\end{minipage}& $\lambda>0$ & $C^{1,1}\big([0,1]\big)$ & Corollary \ref{Cor1.2} a) \\ [1.2ex]  \hline
   $L^\infty(0,1)$ & (\ref{F1})-(\ref{muell}) & $0<\lambda<\lambda_\mu$ & $W^{2,1}(0,1)$ & Theorem \ref{Thm1.3} \\ \hline 
  \end{tabular}
\caption{Overview of the various regularity statements}
\end{table}
Our article is organized as follows: in Section 2 we prove Theorem \ref{Thm1.1} and thus solvability of problem (\ref{J}) and regularity of the unique $W^{1,1}$-minimizer under a rather strong bound on the parameter $\lambda>0$. Section 3 is devoted to the study of the relaxed problem (\ref{K}) where the parameter $\lambda>0$ may be chosen arbitrarily large. The subsequent section deals with a refinement of our regularity result for certain classes of "well behaved" data. Section 5 is devoted to the construction of the counterexample from Theorem \ref{Thm1.5}. Subsequently, we give the proof of Theorem \ref{Thm1.6} where $\mu$-elliptic densities are considered for $\mu\in(1,2]$ and then take a closer look at the Neumann-type boundary value problem (\ref{ode}) from Theorem \ref{Thm1.3} in the seventh section. Finally, we compare our results with a numerically computed example.
\end{section}
\begin{section}{Proof of Theorem \ref{Thm1.1}}\label{section1}
\noindent\textbf{\textit{Proof of part a)}}. Let us assume the validity of the hypotheses of Theorem \ref{Thm1.1}. We first note that  problem (\ref{J}) has at most one solution thanks to the strict convexity of the data fitting quantity $\int\limits_{0}^{1}{(w-f)^2\dt}$ with respect to $w$. Next we show that there exists at least one solution. To this purpose we approximate our original variational problem by a sequence of more regular problems admitting smooth solutions with appropriate convergence properties. This technique is well known from the works \cite{BF1}, \cite{BF2}, \cite{BF3} or \cite{FT}. To become more precise, for fixed $\delta\in(0,1]$ we consider the problem
\begin{align}
\label{Jdelta} J_{\delta}[w]:=\int\limits_{0}^{1}{F_{\delta}(\dot{w})\dt}+\frac{\lambda}{2}\int\limits_{0}^{1}{(w-f)^2\dt}\rightarrow\text{min in}\,\,W^{1,2}(0,1),
\end{align}
where
\begin{align}
 \label{Fdelta} F_{\delta}(p):=\frac{\delta}{2}|p|^2+F(p),\quad p\in\R.
\end{align}
In the following lemma we state that (\ref{Jdelta}) is uniquely solvable in $W^{1,2}(0,1)$ and in addition we will summarize some useful properties of the unique $J_{\delta}$-minimizer $u_{\delta}$. In fact, these results are well-known and have been proved in a much more general setting (see, e.g., \cite{BF1} and \cite{BF2, BF3}).
\begin{Lem}\label{Lem_ud}
The problem (\ref{Jdelta}) admits a unique solution $\ud\in W^{1,2}(0,1)$ for which we have
\begin{enumerate}[a)]
 \item $0\leq\ud\leq 1$ on $[0,1]$,
\item $\ud\in W^{2,2}_{\mathrm{loc}}(0,1)$ (not necessarily uniform in $\delta$),
\item $\sup\limits_{0\leq\delta<1}{\|\ud\|_{W^{1,1}(0,1)}}<\infty$,
\item $\sup\limits_{0\leq\delta<1}{\delta\int\limits_{0}^{1}{|\dot{u}_{\delta}|^2\dt}}<\infty$.
\end{enumerate}
\end{Lem}

\begin{proof}[Proof of Lemma \ref{Lem_ud}]
By the direct method it is immediate that problem (\ref{Jdelta}) has a unique solution $\ud\in W^{1,2}(0,1)$. Since $0\leq f\leq1$ a.e. on $\Omega$, a truncation argument as already carried out in \cite{BF1}, proof of Theorem 1.8 a), (we refer the reader to \cite{BF4} as well) shows $0\leq\ud\leq1$  on $\Omega$, and this proves part a).\\
For part b) we use the well-known difference quotient technique. Observing that we have the uniform estimate $J_{\delta}[\ud]\leq J[0]$ we directly obtain parts c) and d) if we use the definition of $J_{\delta}$ and recall the linear growth of $F$.
\end{proof}

\begin{Rem}
 Note that the results of Lemma \ref{Lem_ud} do not depend on the size of the parameter $\lambda>0$.
\end{Rem}

\begin{Rem}\label{Sobolev_embedding}
In our particular one-dimensional case we emphasize once more that by means of Sobolev's embedding $W^{2,2}_{\mathrm{loc}}(0,1)\hookrightarrow C^1(0,1)$ (see \cite{Ad}) we conclude that $\dot{u}_\delta(t)$ exists for all $t\in(0,1)$ and is continuous.
\end{Rem}

Before starting with the proof of Theorem \ref{Thm1.1} we recall that from the assumptions (\ref{F1})--(\ref{F4}) imposed on the density $F$ and the definition of $\lambda_{\infty}$ (compare (\ref{lambdainf})) it follows
\begin{align}
 \label{Thm14} \text{Im}(F')=(-\lambda_{\infty},\lambda_{\infty}).
\end{align}
Next, we fix $\lambda\in(0,\lambda_{\infty})$ and observe the validity of the following lemma which is of elementary nature but will be important during the further proof.
\begin{Lem}\label{inversefunction}
The inverse function of $F_{\delta}':\R\rightarrow\R$ is uniformly (in $\delta$) bounded on the set $[-\lambda,\lambda]$. 
\end{Lem}

\begin{proof}[Proof of Lemma \ref{inversefunction}]
We observe that $F'$ is an odd, strictly increasing function (compare (\ref{F4})) inducing a diffeomorphism between $\R$ and the open interval $(-\lambda_{\infty},\lambda_{\infty})$. Let us write $(F')^{-1}\big([-\lambda,\lambda]\big)=[-\alpha,\alpha]$ where $F'(\alpha)=\lambda$. Next we choose $t\in[-\lambda,\lambda]$ and assume that $(F_{\delta}')^{-1}(t)>\alpha$. Then it follows (note that $F_{\delta}'$ is strictly increasing)
\[
 t>F_{\delta}'(\alpha)=\delta\alpha+F'(\alpha)=\delta\alpha+\lambda>\lambda,
\]
which is a contradiction. The case $(F_{\delta}')^{-1}(t)<-\alpha$ is treated in the same manner. Thus, the lemma is proved.
\end{proof}

After these preparations we proceed with the proof of Theorem \ref{Thm1.1} a). First, we introduce the continuous functions
\begin{align}\label{defsig}
 \sigma_{\delta}:=F_{\delta}'(\dot{u}_{\delta}).
\end{align}
We wish to note (see, e.g., \cite{FT}) that $\sigma_{\delta}$ is the (unique) solution of the variational problem being in duality to (\ref{Jdelta}) (we will come back to this later in the proof of Theorem \ref{Thm1.4} c)). Using (\ref{F2}) together with Lemma \ref{Lem_ud} d), we obtain 
\begin{align}
 \label{sigmaL2} \sigma_{\delta}\in L^2(0,1)\quad\text{uniformly in}\,\,\delta.
\end{align}
Next, we observe that $\ud$ solves the Euler equation
\begin{align}
\label{eueq} 0=\int\limits_{0}^{1}{F_{\delta}'(\dot{u}_{\delta})\dot{\varphi}\dt}+\lambda\int\limits_{0}^{1}{(\ud-f)\varphi \dt} 
\end{align}
for all $\varphi\in W^{1,2}(0,1)$. Note, that by (\ref{defsig}) this equation states that  $\sigma_{\delta}$ is weakly differentiable with
\begin{align}
 \label{Thm11} \dot{\sigma}_{\delta}=\lambda(\ud-f)\quad\text{a.e. on}\,\,(0,1).
\end{align}
Combining Lemma \ref{Lem_ud} a) with (\ref{sigmaL2}) and (\ref{Thm11}) it follows (recall our assumption $0\leq f\leq1$ a.e. on $(0,1)$)
\begin{align}
 \label{Thm12} \sigma_{\delta}\in W^{1,\infty}(0,1)=C^{0,1}\big([0,1]\big)\quad\text{uniformly in}\,\,\delta\,\,\text{and}\,\,\|\dot{\sigma}_{\delta}\|_{\infty}\leq\lambda.
\end{align}
Choosing $\varphi\in C^1\big([0,1]\big)$ in (\ref{eueq}) and recalling (\ref{Thm11}) it holds (see \cite{HS}, (18.16) Theorem, p. 285 or \cite{BGH}, Chapter 2)
\begin{align*}
0=&\int\limits_{0}^{1}{\big(\dot{\sigma}_{\delta}\varphi+\sigma_{\delta}\dot{\varphi}\big)\dt}=\int\limits_{0}^{1}{\frac{d}{dt}\big(\sigma_{\delta}\varphi\big)\dt}=\sigma_{\delta}(1)\varphi(1)-\sigma_{\delta}(0)\varphi(0).
\end{align*}
Thus, since $\varphi\in C^1\big([0,1]\big)$ is arbitrary it must hold
\begin{align}
 \label{Thm13} \sigma_{\delta}(0)=\sigma_{\delta}(1)=0.
\end{align}
Note that (\ref{Thm12}) and (\ref{Thm13}) imply
\begin{align}
 \label{Thm131} \|\sigma_{\delta}\|_{\infty}\leq\lambda.
\end{align}
At this point, the definition of $\sigma_{\delta}$, (\ref{Thm12}), (\ref{Thm13}), (\ref{Thm131}) and Lemma \ref{inversefunction} yield existence of a constant $M>0$, independent of $\delta$, such that
\begin{align}
 \label{Thm16} \|\dot{u}_{\delta}\|_{\infty}\leq M.
\end{align}
Here we have made essential use of the restriction $\lambda<\lambda_{\infty}$. As a consequence, there exists a function $u\in W^{1,\infty}(0,1)$ such that $\ud\rightrightarrows u$ uniformly as $\delta\downarrow0$ and $\dot{u}_{\delta}\rightharpoondown\dot{u}$ in $L^p(0,1)$ for all finite $p>1$ as $\delta\downarrow0$, at least for a subsequence. Now, our goal is to show that $u$ is $J$-minimal: thanks to the $J_{\delta}$-minimality of $\ud$ it follows for all $v\in W^{1,2}(0,1)$
\[
 J[\ud]\leq J_{\delta}[\ud]\leq J_{\delta}[v]\overset{\delta\downarrow0}{\rightarrow} J[v]
\]
together with 
\[
 J[u]\leq\liminf\limits_{\delta\rightarrow0}{J[\ud]}.
\]
Thus, we have $J[u]\leq J[v]$ for all $v\in W^{1,2}(0,1)$ and from this we get $J[u]\leq J[w]$ for all $w\in W^{1,1}(0,1)$ by approximating $w$ with a sequence $(v_k)\subset W^{1,2}(0,1)$ in the $W^{1,1}$-topology. This finally proves  $u$ to be a solution of problem (\ref{J}).  This proves part a). We continue with the

\noindent{\textbf{\textit{proof of part b)}}}. Considering the relaxed variant $K$ from (\ref{K}) of the functional $J$, it is easy to check that $K$ has a unique solution $\widetilde{u}\in BV(0,1)$, compare the comments given in the beginning of the proof of Theorem \ref{Thm1.2} a). This, together with the $J$-minimality of $u$, implies $K[\widetilde{u}]\leq J[u]$ since it holds $K[w]=J[w]$ for all functions $w\in W^{1,1}(0,1)$. To show the reverse inequality we note, that we can approximate $\widetilde{u}\in BV(0,1)$ by a sequence of smooth functions $(u_n)\subset C^{\infty}(0,1)$ such that (as $n\rightarrow\infty$)
\[
 u_n\rightarrow \widetilde{u}\quad\text{in}\,\,L^1(0,1)\quad\text{and}\,\,\int\limits_{0}^{1}{\sqrt{1+\dot{u}_n^2}\,dt}\rightarrow\int\limits_{0}^{1}{\sqrt{1+|D\widetilde{u}|^2}}
\]
(see e.g. \cite{AG}, Proposition 2.3), where $\int\limits_{0}^{1}{\sqrt{1+|D\widetilde{u}|^2}}$ denotes the total variation of the vector measure $(\mathcal{L}^1,Du)^T$. Note that we even have $u_n\rightarrow \widetilde{u}\quad\text{in}\,\,L^p(0,1)$ for any finite $p>1$ by the $BV$-emebdding theorem. Now it is well-known that the functional $K$ is continuous with respect to the above notion of convergence (see e.g. \cite{AG}, Proposition 2.2) and it follows
\[
 K[\widetilde{u}]=\lim\limits_{n\rightarrow\infty}{K[u_n]}=\lim\limits_{n\rightarrow\infty}{J[u_n]}\geq J[u].
\]
Hence, $K[\widetilde{u}]=J[u]$, i.e. $u$ is $K$-minimal and it holds $u=\widetilde{u}$ due to the uniqueness of the $K$-minimizer.  Finally, we give the

\noindent\textbf{\textit{proof of part c)}}. By (\ref{eueq}) and Lemma \ref{Lem_ud} b) it holds
\[
 \ddot{u}_{\delta}=\lambda\frac{(\ud-f)}{F_{\delta}''(\dot{u}_{\delta})}\quad\text{a.e. on}\,\,(0,1),
\]
hence $\dot{u}_\delta\in W^{1,\infty}(0,1)$ uniformly in $\delta$ on account of (\ref{Thm16}). Thus the functions $\dot{u}_\delta$ have a unique Lipschitz extension to the boundary points $0$ and $1$, which in particular implies the differentiability of $u_\delta$ at $0$ and $1$ with values of the derivatives given by the values of the Lipschitz extension of $\dot{u}_\delta$. Thus there is a clear meaning of $\dot{u}_\delta(0)$ and $\dot{u}_\delta(1)$. By continuity reasons the defining equation (\ref{defsig}) for $\sigma_\delta$ extends to the boundary points of $(0,1)$ and since $F'_\delta$ vanishes exactly in the origin, it follows from (\ref{Thm13}) that $\dot{u}_\delta(0)=\dot{u}_\delta(1)=0$. Combining this with the uniform boundedness of $u_\delta$ in $C^{1,1}\big([0,1]\big)$, we immediately see that $u\in C^{1,1}\big([0,1]\big)$ holds together with the boundary condition $\dot{u}(0)=\dot{u}(1)=0$. Furthermore, $u$ solves the Euler equation
\[
 0=\int\limits_{0}^{1}{F'(\dot{u})\dot{\varphi}\dt}+\lambda\int\limits_{0}^{1}{(u-f)\varphi \dt}
\]
for all $\varphi\in C^1_0(0,1)$ and from this we conclude the validity of the relation
\[
 \frac{d}{dt}F'(\dot{u})=\lambda(u-f)\quad\text{a.e. on}\,\,(0,1).
\]
Consequently, we have
\[
 \ddot{u}=\lambda\frac{u-f}{F''(\dot{u})}\quad\text{a.e. on}\,\,(0,1),
\]
together with $\dot{u}(0)=\dot{u}(1)=0$, i.e. $u$ solves the boundary value problem (\ref{ode}), which was the statement of part c). \qed
\end{section}

\begin{section}{Proof of Theorem \ref{Thm1.2}}\label{section2}
Let us assume the validity of the hypotheses of Theorem \ref{Thm1.2}. We start with the

\noindent\textbf{\textit{proof of part a)}}. That in fact the functional $K$ from (\ref{K}) admits a unique minimizer $u\in BV(0,1)$ is straightforward in the framework of the theory of $BV$-functions (see e.g. \cite{AFP}, Theorem 3.23, p. 132 as well as  Remark 5.46 and Theorem 5.47 on p. 303/304). The justification that we have $0\leq u\leq 1$ a.e. on $(0,1)$ follows by a truncation argument (see \cite{BF3} in the case of pure denoising and \cite{BF4}). For later purposes we like to show that the minimizer $u$ can also be obtained as the limit of the regularizing sequence introduced in Lemma \ref{Lem_ud} giving $0\leq u\leq 1$ as a byproduct of Lemma \ref{Lem_ud} a): as done there, we study the problem
\[
 J_{\delta}[w]:=\frac{\delta}{2}\int\limits_{0}^{1}{|\dot{w}|^2dx}+ J[w], \quad w\in W^{1,2}(0,1),
\]
where in particular it holds $0\leq\ud\leq1$ for all $t\in[0,1]$ (see Lemma \ref{Lem_ud}, a)). Next we show that $\ud\rightarrow u$ in $L^1(0,1)$ and a.e. at least for a subsequence. First, by Lemma \ref{Lem_ud} c), there exists $\widetilde{u}\in BV(0,1)$ such that (for a subsequence) $\ud\rightarrow\widetilde{u}$ in $L^1(0,1)$. By lower semicontinuity we have 
\[
 K[\widetilde{u}]\leq \liminf\limits_{\delta\downarrow0}{J[\ud]},
\]
which yields by using the $K$-minimality of $u$
\[
 K[u]\leq  K[\widetilde{u}]\leq\liminf\limits_{\delta\downarrow0}{J_{\delta}[\ud]}.
\]
As in the proof of Theorem \ref{Thm1.1} b) we approximate the function $u$ by a sequence of smooth functions $(u_m)\subset C^{\infty}(0,1)$ such that (as $m\rightarrow\infty$)
\[
 u_m\rightarrow u\quad\text{in}\,\,L^1(0,1),\quad \int\limits_{0}^{1}{\sqrt{1+\dot{u}_m^2}\dt}\rightarrow\int\limits_{0}^{1}{\sqrt{1+|Du|^2}},
\]
and observe $u_m\rightarrow u\quad\text{in}\,\,L^p(0,1)$ for each finite $p>1$. Since $K$ is continuous with respect to the above notion of convergence  we obtain $K[u_m]\rightarrow K[u]$ as $m\rightarrow\infty$. This implies by using the $J_{\delta}$-minimality of $\ud$
\[
 K[u]\leq  K[\widetilde{u}]\leq\liminf\limits_{\delta\downarrow0}{J_{\delta}[\ud]}\leq \liminf\limits_{\delta\downarrow0}{J_{\delta}[u_m]}=J[u_m]=K[u_m].
\]
Thus, after passing to the limit $m\rightarrow\infty$, it follows
\[
  K[u]\leq  K[\widetilde{u}]\leq K[u],
\]
which implies $u=\widetilde{u}$ by the uniqueness of the $K$-minimizer and hence $0\leq u\leq 1$ a.e. on $(0,1)$. 

\noindent\textbf{\textit{Proof of part b)}}. With $\sigma_{\delta}$ as defined in the proof of Theorem \ref{Thm1.1} (see (\ref{defsig})), we recall that we have (\ref{Thm11})--(\ref{Thm13}) at hand. Note that at this stage no bound on $\lambda$ was necessary. Thus, there exists $\sigma\in W^{1,\infty}(0,1)$ with $\sigma_{\delta}\rightrightarrows\sigma$ as $\delta\downarrow0$ (at least for a subsequence). Moreover
\begin{align}
\label{Thm121} \begin{cases}
\dot{\sigma}=\lambda(u-f)\text{ and thus }|\dot{\sigma}(t)|\leq\lambda\quad\text{a.e.},\\
|\sigma(t)|\leq\lambda\quad\text{on}\,\,[0,1],\\
\sigma(0)=\sigma(1)=0.               
\end{cases}
\end{align}
In accordance with \cite{FT}, Theorem 1.3 (in the case of pure denoising), $\sigma$ is the unique solution of the dual problem associated to (\ref{J}) and it holds
\begin{align}
 \label{Thm122} \sigma=F'(\dot{u})\quad\text{a.e.},
\end{align}
where $u$ is the unique solution of problem (\ref{K}) in the class $BV(0,1)$ and $\dot{u}$ in the following denotes the Lebesgue point representative of the density of the absolutely continuous part $D^au$ of the measure $Du$. Thus, there is a null set $A\subset(0,1)$ such that we have (see (\ref{Thm122}))
\begin{align}
 \label{Thm123} \sigma(t)=F'(\dot{u}(t)),\quad t\in(0,1)-A.
\end{align}
Let us fix $t_0\in(0,1)-A$. Then it holds $|\sigma(t_0)|<\lambda_{\infty}$ and since $\sigma$ is continuous (recall (\ref{Thm121})), there exists $\eps>0$ with
\begin{align}
 \label{Thm124} |\sigma(t)|\leq \lambda_{\infty}-\alpha\quad\text{for all}\,\,t\in[t_0-\eps,t_0+\eps],
\end{align}
where $\alpha>0$ is chosen appropriately. Recalling $\sigma_{\delta}\rightrightarrows\sigma$, (\ref{Thm124}) yields for $\delta\leq\delta_{\eps}$
\begin{align}
 \label{Thm125} |\sigma_{\delta}(t)|\leq \lambda_{\infty}-\frac{\alpha}{2}\quad\text{for all}\,\,t\in[t_0-\eps,t_0+\eps].
\end{align}
Quoting Lemma \ref{inversefunction}, $(F'_{\delta})^{-1}$ is uniformly (with respect to $\delta$) bounded on the interval $\big[-\lambda_{\infty}+\frac{\alpha}{2},\lambda_{\infty}-\frac{\alpha}{2}\big]$. Hence, there exists a number $L>0$, independent of $\delta$, such that (compare (\ref{Thm16}))
\begin{align}
 \label{Thm126} \|\dot{u}_{\delta}\|_{L^{\infty}(t_0-\eps,t_0+\eps)}\leq L \text{ for all } \eps\leq\delta.
\end{align}
Since $u$ is the $L^1$-limit of the sequence $(\ud)$ (compare the proof of part a) of this theorem), (\ref{Thm126}) ensures
\[
 u\in C^{0,1}\big([t_0-\eps,t_0+\eps]\big).
\]
Further using the Euler equation (\ref{eueq}) for $\ud$ on $(t_0-\eps,t_0+\eps)$ we deduce
\[
 \ddot{u}_{\delta}=\lambda\frac{(\ud-f)}{F_{\delta}''(\dot{u}_{\delta})}\quad\text{a.e. on}\,\,(t_0-\eps,t_0+\eps),
\]
which yields the existence of a number $L'>0$, independent of $\delta$, such that
\begin{align}
\label{Thm127} \|\ddot{u}_{\delta}\|_{L^{\infty}(t_0-\eps,t_0+\eps)}\leq L'.
\end{align}
From (\ref{Thm127}), it finally follows
\[
 u\in C^{1,1}\big([t_0-\eps,t_0+\eps]\big),
\]
and this shows that $u$ is of class $C^{1,1}$ in a neighbourhood of a point $t\in (0,1)$ if and only if $t$ is a Lebesgue point of $\dot{u}$.
Recalling (\ref{Thm121}) we can conclude that (\ref{Thm124}) (which by the way implies (\ref{Thm126}) and (\ref{Thm127})) is true on a suitable interval $[0,t_1]$. This can be achieved by setting $t_1<\sup\{s\in[0,1]:\,|\sigma(s)|<\lambda_{\infty}\}$, for instance. Hence, $u\in C^{1,1}\big([0,t_1]\big)$. Using analogous arguments we can show existence of a number $t_2$ for which we have $0<t_1\leq t_2<1$ and such that $u\in C^{1,1}([t_2,1])$. This proves part b) of the theorem.

\noindent\textbf{\textit{Proof of part c)}}. Our strategy is to prove $\ud\in W^{1,2}_{\text{loc}}(a,b)$ uniformly with respect to $\delta$. With this result at hand along with the fact that the $K$-minimizing function $u\in BV(0,1)$ is obtained as the limit of the sequence $(\ud)$, we see that $u\in BV(a,b)\cap W^{1,2}_{\text{loc}}(a,b)$, thus $u\in W^{1,1}(a,b)$. First, we recall $\ud\in W^{2,2}_{\text{loc}}(0,1)$ (compare Lemma \ref{Lem_ud}) and $F'_{\delta}(\dot{u}_{\delta})$ is of class $W^{1,2}_{\text{loc}}(0,1)$ satisfying
\[
 (F'_{\delta}(\dot{u}_{\delta}))'=F''_{\delta}(\dot{u}_{\delta})\ddot{u}_{\delta}\quad\text{a.e. on}\,\,(0,1).
\]
From (\ref{eueq}) we therefore get
\begin{align}
\label{Thm128} \int\limits_{0}^{1}{F''_{\delta}(\dot{u}_{\delta})\ddot{u}_{\delta}\dot{\varphi}\dt}=\lambda\int\limits_{0}^{1}{(\ud-f)\dot{\varphi}\dt}
\end{align}
for all $\varphi\in C^{\infty}_0(0,1)$ and by approximation, (\ref{Thm128}) remains valid for functions $\varphi\in W^{1,2}(0,1)$ that are compactly supported in $(0,1)$.
Next, we fix a point $x_0\in(a,b)$, a number $R>0$ such that $(x_0-2R,x_0+2R)\Subset(a,b)$ and $\eta\in C^{\infty}_0(x_0-2R,x_0+2R)$ with $\eta\equiv1$ on $(x_0-R,x_0+R)$, $0\leq\eta\leq1$ as well as $|\dot{\eta}|\leq\frac{c}{R}$. We choose $\varphi:=\eta^2\dot{u}_{\delta}$ in (\ref{Thm128}) and obtain
\begin{align}
 \label{Thm129}\begin{split} I_0:&=\int\limits_{x_0-2R}^{x_0+2R}{F''_{\delta}(\dot{u}_{\delta})(\ddot{u}_{\delta})^2\eta^2\dt}\\
&=-2\int\limits_{x_0-2R}^{x_0+2R}{F''_{\delta}(\dot{u}_{\delta})\ddot{u}_{\delta}\dot{u}_{\delta}\dot{\eta}\eta \dt}+\lambda\int\limits_{x_0-2R}^{x_0+2R}{(\ud-f)\dot{\varphi}\dt}\\
&=:I_1+\lambda I_2.              
\end{split}
\end{align}
We start with estimating $I_1$ where, by using Young's inequality for fixed $\eps>0$, we get
\begin{align}
\label{Thm1210} |I_1|\leq\eps I_0+c\eps^{-1}\int\limits_{x_0-2R}^{x_0+2R}{F''_{\delta}(\dot{u}_{\delta})\dot{u}_{\delta}^2\dot{\eta}^2\dt}.
\end{align}
An integration by parts (recall $f\in W^{1,2}_{\text{loc}}(a,b)$) further gives for $I_2$
\begin{align}
 \label{Thm1211} I_2=-\int\limits_{x_0-2R}^{x_0+2R}{(\dot{u}_{\delta}-\dot{f})\dot{u}_{\delta}\eta^2\dt}=-\int\limits_{x_0-2R}^{x_0+2R}{\dot{u}_{\delta}^2\eta^2\dt}+\int\limits_{x_0-2R}^{x_0+2R}{\dot{f}\dot{u}_{\delta}\eta^2\dt}.
\end{align}
Putting together (\ref{Thm1210}) and (\ref{Thm1211}) and absorbing terms (we choose $\eps>0$ sufficiently small), (\ref{Thm129}) implies
\begin{align}
 \label{Thm1212} \begin{split}
&\int\limits_{x_0-2R}^{x_0+2R}{F''_{\delta}(\dot{u}_{\delta})(\ddot{u}_{\delta})^2\eta^2\dt}+\lambda \int\limits_{x_0-2R}^{x_0+2R}{\dot{u}_{\delta}^2\eta^2\dt}\\
&\leq c\int\limits_{x_0-2R}^{x_0+2R}{F''_{\delta}(\dot{u}_{\delta})\dot{u}_{\delta}^2\dot{\eta}^2\dt}+c\int\limits_{x_0-2R}^{x_0+2R}{|\dot{f}||\dot{u}_{\delta}|\eta^2\dt}.                  
\end{split}
\end{align}
The first integral on the right-hand side of (\ref{Thm1212}) can be handled by the uniform estimate $J_{\delta}[\ud]\leq J[0]$, the linear growth of $F$ and condition (\ref{F5}). More precisely we get
\[
 \int\limits_{x_0-2R}^{x_0+2R}{F''_{\delta}(\dot{u}_{\delta})\dot{u}_{\delta}^2\dot{\eta}^2\dt}\leq c(R)\int\limits_{x_0-2R}^{x_0+2R}{(\delta+(1+\dot{u}_{\delta}^2)^{-\frac{1}{2}})\dot{u}_{\delta}^2\dt}\leq c(R),
\]
where $c(R)$ denotes a local constant being independent of $\delta$. To the second integral  we apply Young's inequality ($\eps>0$) which yields
\[
 \int\limits_{x_0-2R}^{x_0+2R}{|\dot{f}||\dot{u}_{\delta}|\eta^2\dt}\leq \eps\int\limits_{x_0-2R}^{x_0+2R}{\dot{u}_{\delta}^2\eta^2\dt}+c\eps^{-1}\int\limits_{x_0-2R}^{x_0+2R}{\dot{f}^2\eta^2\dt}
\]
Absorbing terms by choosing $\eps>0$ sufficiently small, (\ref{Thm1212}) yields (recall $\eta\equiv1$ on $(x_0-R,x_0+R)$ and $f\in W^{1,2}_{\text{loc}}(a,b)$ once again)
\begin{align}
 \label{Thm1213} \int\limits_{x_0-R}^{x_0+R}{F''_{\delta}(\dot{u}_{\delta})(\ddot{u}_{\delta})^2\dt}+\lambda \int\limits_{x_0-R}^{x_0+R}{\dot{u}_{\delta}^2\dt}\leq c(f,R),
\end{align}
where $c(f,R)$ is a local constant, independent of $\delta$. This proves
\[ 
\ud\in W^{1,2}(x_0-R,x_0+R) \text{ uniformly with respect to $\delta$ }
\]
and part c) of the theorem now follows from a covering argument. \qed

\begin{Rem}\label{singularsigma}
From the proof of part b) we see how the singular set $\mathrm{Sing}(u):=[0,1]-\mathrm{Reg}(u)$ can be given in terms of $\sigma$: due to (\ref{Thm122}), we have $|\sigma(t)|<\lambda_\infty$ at almost all points $t\in[0,1]$ and thus, since $\sigma$ is continuous it holds
\[
 -\lambda_\infty\leq\sigma(t)\leq\lambda_\infty\;\text{ for all $t\in[0,1]$}.
\]
We claim that $\mathrm{Sing}(u)$ is exactly the set of points where $|\sigma|$ attains the maximal value $\lambda_\infty$, i.e.
\[
 \mathrm{Sing}(u)=\{t\in[0,1]\,:\,|\sigma(t)|=\lambda_\infty\}.
\]
Indeed, let $t_0\in[0,1]$ be a regular point of $u$, i.e. there is a small neighbourhood $(t_0-\eps,t_0+\eps)$ of $t_0$ such that $u$ is of class $C^{1,1}(t_0-\eps,t_0+\eps)$. Hence $|\dot{u}|$ is bounded on $(t_0-\eps,t_0+\eps)$ and  (\ref{Thm123}) along with the continuity of $\sigma$ implies $|\sigma(t_0)|<\lambda_\infty$. Conversely, if $s_0\in[0,1]$ is a point where $|\sigma(s_0)|<\lambda_\infty$ the arguments after (\ref{Thm123}) show that $s_0$ is a regular point.
\end{Rem}

\end{section}

\begin{section}{Proof of Theorem \ref{Thm1.4}}
\noindent\textbf{\textit{Proof of part a)}} Without loss of generality we will in the following identify $f$ with the representative that is continuous in $t_0$. Moreover, we recall that we consider the ``good`` representative of $u$ as specified in the introduction around the formula (\ref{Du}). Assume that the statement is false, i.e. the left- and the right limit of $u$ at $t_0$,
\begin{align*}
 u^-(t_0):=\lim\limits_{t_k\uparrow t_0}u(t_k),\quad u^+(t_0):=\lim\limits_{t_k\downarrow t_0}u(t_k)
\end{align*}
do not coincide. We may assume
\begin{align}\label{assumption}
 u^-(t_0)<f(t_0)\quad\text{ and }\quad u^+(t_0)\geq f(t_0),
\end{align}
and it will be clear from the proof, that all the other possible cases can be treated analogously. Let $h_0:=u^+(t_0)-u^-(t_0)$ denote the jump-height at $t_0$. Then, from (\ref{assumption}) it follows in particular that there exist $\eps>0$ and $0<d<h_0$ such that
\begin{align*}
 u(t)<f(t)-d\text{ for all }t\in [t_0-\eps,t_0].
\end{align*}
We may further assume that $u$ is continuous at $t_0-\eps$. Now define $\tilde{u}$ by 
\begin{align*}
 \tilde{u}(t):=u(t)+d\chi_{[t_0-\eps,t_0]}(t).
\end{align*}
That means, on $[t_0-\eps,t_0]$ we ''move'' $u$ a little closer to $f$ so that in particular
\begin{align}\label{errorterm}
 \intop_0^1(\tilde{u}-f)^2\,dt<\intop_0^1(u-f)^2\,dt.
\end{align}
Let us write (compare (\ref{Du})) $Du=\dot u\mathcal{L}^1+\sum_{k=0}^\infty h_k\delta_{x_k}+D^cu$, where $\{t_k\}_{k=0}^\infty$ is the jump-set of $u$. Clearly, $\tilde{u}\in BV(0,1)$ and it holds
\begin{align*}
 D\tilde{u}=\dot{u}\mathcal{L}^1+(h_0-d)\delta_{t_0}+d\delta_{t_0-\eps}+\sum_{k=1}^\infty h_k\delta_{x_k}+D^cu
\end{align*}
and in conclusion
\begin{align*}
K[\tilde{u}]=\intop_0^1F(\dot{u})\dt&+\lambda_\infty\left(|h_0-d|+d+\sum_{k=1}^\infty |h_k|\right)\\
&+\lambda_\infty|D^cu|(0,1)+\frac{\lambda}{2}\intop_0^1(\tilde{u}-f)^2\,dt.
\end{align*}
Since $d<h_0$ and due to (\ref{errorterm}) this implies
\[
 K[\tilde{u}]<K[u],
\]
in contradiction to the minimality of $u$.

\noindent\textbf{\textit{Proof of part b)}}. First we notice, that due to Theorem \ref{Thm1.2} part b) there are $s_1$ and $s_2$ in $(a,b)$, arbitrarily close to $a$ and $b$ respectively with $s_1<s_2$ and such that $u$ is $C^{1,1}$-regular in a small neighbourhood of $s_1$ and $s_2$. Hence, the singular set 
\[
 S:=\text{Sing}(u)\cap[s_1,s_2]
\]
is a compact subset of $(s_1,s_2)$. Moreover, by part a) of Theorem \ref{Thm1.4} we have $u\in C^0(a,b)$. Assume $S\neq\emptyset$. Then there exists $\overline{s}:=\inf S>a$ which is an element of $S$ itself since the singular set is closed. In particular, $\sigma(\overline{s})=\pm\lambda_\infty$ (cf. Remark \ref{singularsigma}), i.e. $\sigma$ has a maximum respectively minimum in $\overline{s}$ and since $\dot{\sigma}=\lambda (u-f)\in C^0(a,b)$ it follows 
\[
 \dot{\sigma}(\overline{s})=0
\]
which means 
\begin{align}\label{u=f}
 u(\overline{s})=f(\overline{s}).
\end{align}
Without loss of generality we may assume $\sigma(\overline{s})=\lambda_\infty$. Since $\sigma$ is continuous in $\overline{s}$, for any sequence $t_k\uparrow \overline{s}$ approaching $\overline{s}$ from the left it must hold $\sigma(t_k)\rightarrow \lambda_\infty$ and thus, because of $\dot{u}=DF^{-1}(\sigma)$,
\begin{align}\label{stern}
 \dot{u}(t_k)\rightarrow\infty\;\text{ for any sequence }t_k\uparrow \overline{s}.
\end{align}
In particular, for arbitrary $M>0$ there exists $\eps>0$ such that 
\begin{align}\label{M}
\dot{u}(t)>M \text{ for } t\in [\overline{s}-\eps,\overline{s}).
\end{align}
Now choose $M:=\|\dot{f}\|_{\infty;[s_1,s_2]}$ in (\ref{M}). Then $\frac{d}{dt}(u-f)>0$ on $[\overline{s}-\eps,\overline{s})$, which is not compatible with (\ref{u=f}) unless $u-f<0$ on $[\overline{s}-\eps,\overline{s})$. But in this case, the differential equation
\begin{align}
 \ddot{u}=\lambda\frac{u-f}{F''(\dot{u})}\text{ a.e. on }[\overline{s}-\eps,\overline{s})\label{difgl}
\end{align}
  implies that $\dot{u}$ is strictly decreasing on $[\overline{s}-\eps,\overline{s})$ and thereby $\dot{u}(\overline{s}-\eps)\geq \dot{u}(s)$ for all $s\in[\overline{s}-\eps,\overline{s})$ which is inconsistent with (\ref{stern}). This shows $\mathrm{Sing}(u)\cap (a,b)=\emptyset$ by contradiction and hence $u\in C^1(a,b)$. Moreover, since $\sigma$ is locally bounded away from $\lambda_\infty$ we even have $u\in W^{2,\infty}_{\mathrm{loc}}(a,b)$. Hence (\ref{difgl}) holds at almost all points of $(a,b)$ and by the continuity of $\dot{u}$, the right-hand side of (\ref{difgl}) is continuous. It therefore follows that $u\in C^2(a,b)$.
We proceed with the 

\noindent\textbf{\textit{proof of part c)}}. As already mentioned, the auxiliary quantity $\sigma$ that has been introduced in the proof of Theorem \ref{Thm1.2} has an independent meaning as the solution of the dual problem to $J\rightarrow\min$. As e.g. in \cite{BF3} or \cite{FT}, we obtain the dual problem from the Lagrangian given by
\begin{align*}
 L(v,\kappa):=\intop_{0}^1\kappa\dot{v}\dt-\intop_{0}^1F^*(\kappa)\dt+\underset{\mbox{$=:\Psi(v)$}}{\underbrace{\frac{\lambda}{2}\intop_0^1(v-f)^2\dt}},
\end{align*}
where $(v,\kappa)\in W^{1,1}(0,1)\times L^\infty(0,1)$,
\[
F^*(\kappa):=\sup\limits_{w\in L^1(0,1)}\big(\langle\kappa,w\rangle-F(w)\big)
\]
 is the  convex conjugate and 
\[
\langle\kappa,w\rangle:=\intop_0^1\kappa w\dt
\]
denotes the duality product of $L^1(0,1)$ and $L^\infty(0,1)$. By standard results from convex analysis (see e.g. \cite{ET}, Remark 3.1 on p. 56), the functional $J$ can be expressed in terms of the Lagrangian by $J[v]=\sup\limits_{\kappa\in L^\infty(0,1)}L(v,\kappa)$ and 
\begin{align*}
 R[\kappa]:=\inf\limits_{v\in W^{1,1}(0,1)}L(v,\kappa),\;\kappa\in L^\infty(0,1)
\end{align*}
is called the dual functional. The dual problem consists in maximizing $R[\kappa]$ in $L^\infty(0,1)$. Obviously, $\kappa:=\sigma$ is an admissible choice and since $\sigma\in W^{1,\infty}(0,1)$ along with $\sigma(0)=F'(\dot{u}(0))=0=F'(\dot{u}(1))=\sigma(1)$ (cf. (\ref{Thm121})), we can integrate by parts and derive the following integral representation of the dual functional (cf. also Theorem 9.8.1 on p. 366 in \cite{ABM}):
\begin{align*}
 R[\sigma]&=\inf_{v\in W^{1,1}(0,1)}\intop_{0}^1\sigma\dot{v}\dt-\intop_{0}^1F^*(\sigma)\dt+\Psi(v)\\
 &=-\intop_{0}^1F^*(\sigma)\dt-\sup_{v\in W^{1,1}(0,1)}\left(-\intop_{0}^1\sigma\dot{v}\dt-\Psi(v)\right)\\
 &=-\intop_{0}^1F^*(\sigma)\dt-\sup_{v\in W^{1,1}(0,1)}\left(\intop_{0}^1\dot{\sigma}v\dt-\Psi(v)\right)\\
 &=-\intop_{0}^1F^*(\sigma)\dt-\Psi^*(\dot{\sigma}).
\end{align*}
Next, we want to compute  $\Psi^*(\dot{\sigma})$. By definition we have
\begin{align*}
 \Psi^*(\dot{\sigma})&=\sup_{v\in W^{1,1}(0,1)}\left(\langle v,\dot{\sigma}\rangle-\frac{\lambda}{2}\langle v-f,v-f\rangle\right)\\
 &= \sup_{v\in W^{1,1}(0,1)}\left\langle v,\dot{\sigma} -\frac{\lambda}{2}v+\lambda f\right\rangle-\frac{\lambda}{2}\langle f,f\rangle.\\   
\end{align*}
Applying Hölder's inequality, we get
\begin{align}\label{holder}
 \left\langle v,\dot{\sigma} -\frac{\lambda}{2}v+\lambda f\right\rangle\leq -\frac{\lambda}{2}\|v\|_{2}^2+\|\dot{\sigma}+\lambda f\|_{2}\|v\|_2
\end{align}
and by elementary calculus, the right-hand side is maximal for $\|v\|_2=\|\frac{\dot{\sigma}}{\lambda}+f\|_2$. An easy computation confirms, that for the choice $v=\frac{\dot{\sigma}}{\lambda}+f$ the  left-hand side of (\ref{holder}) attains this maximal value and it follows
\begin{align*}
 \Psi^*(\dot{\sigma})&=\intop_{0}^1\left(\frac{\dot{\sigma}}{\lambda}+f\right)\dot{\sigma}\dt-\frac{\lambda}{2}\intop_0^1\left(\frac{\dot{\sigma}}{\lambda}+f\right)^2\dt\\
 &=\intop_0^1\frac{\dot{\sigma}^2}{2\lambda}+\dot{\sigma} f\dt.
\end{align*}
Thereby we obtain for $R[\sigma]$
\begin{align}\label{intrepofR}
 R[\sigma]=-\intop_0^1\frac{\dot{\sigma}^2}{2\lambda}+\dot{\sigma} f\dt-\intop_0^1F^*(\sigma)\dt.
\end{align}
Now assume that $\mathrm{Sing}(u)\neq 0$. By Remark \ref{singularsigma}, this means that there exists at least one point $t\in [0,1]$ where $\sigma(t)=\pm\lambda_\infty$. Let $\hat{t}$ denote the smallest such $t$. Since $\sigma(0)=0$ it follows $\hat{t}>0$ and without loss of generality we may assume $\sigma(\hat{t})=\lambda_\infty$. Let $\varphi\in C_0^\infty\big([0,\hat{t})\big)$ be an arbitrary test function. On $[0,\hat{t})$ it holds $|\sigma|<\lambda_\infty$ and since $\mathrm{spt}\,\varphi$ is a compact subset of $[0,\hat{t})$ (and $\sigma$ is continuous) there exists $\eps_0=\eps_0(\varphi)$ such that $|\sigma(t)+\eps\varphi(t)|\leq\lambda_\infty-\delta$ for some $\delta>0$ and for all $0\leq \eps<\eps_0$. By Theorem 26.4  and Corollary 26.4.1 in \cite{Ro}, $F^*$ is  finite and continuously differentiable on $(-\lambda_\infty,\lambda_\infty)$ (with derivative $(F^*)'=(F')^{-1}$) and hence
\[
 \left.\frac{d}{d\eps}\right|_{\eps=0}F^*(\sigma(t)+\eps\varphi(t))=(F^*)'(\sigma(t))\varphi(t)\in L^1(0,\hat{t}),
\]
which together with (\ref{intrepofR}) and the maximality of $\sigma$ implies that the following Euler equation must hold for all $\varphi\in C^\infty_0(0,\hat{t})$:
\begin{align}\label{eulersigma}
 -\intop_0^1\frac{\dot{\sigma}}{\lambda}\dot{\varphi}+f\dot{\varphi}\dt-\intop_0^1(F^*)'(\sigma)\varphi\dt=0.
\end{align}
Since $[0,\hat{t})\subset\mathrm{Reg}(u)$ and $f\in W^{1,1}(0,1)$ by assumption, we have (see (\ref{Thm121}))
\begin{align}\label{sigmadot}
\dot{\sigma}=\lambda(u-f)\in W^{1,1}(0,\hat{t})
\end{align}  
and therefore $\sigma\in W^{2,1}(0,1)$, so that (\ref{eulersigma}) implies the following differential equation:
\begin{align}\label{dglsigma}
 \frac{\ddot{\sigma}}{\lambda}+\dot{f}-(F^*)'(\sigma)=0\,\text{ a.e. on }(0,\hat{t}).
\end{align}
Let $\{s_k\}\subset[0,\hat{t})$, $k\in\N$ denote a sequence with $s_k\uparrow\hat{t}$ as $k\rightarrow\infty$. Multiplying (\ref{dglsigma}) with $\dot{\sigma}$ and integrating by parts (recall $\dot{\sigma}\in W^{1,1}(0,\hat{t})$) then yields
\[
 \frac{\dot{\sigma}(s_k)^2}{2\lambda}-\frac{\dot{\sigma}(0)^2}{2\lambda}+\intop_0^{s_k}\dot{f}\dot{\sigma}\dt-F^*(\sigma_{s_k})=0.
\]
Since $\dot{\sigma}$ is bounded by $\lambda$, this implies the estimate
\begin{align}\label{F*}
 F^*(\sigma(s_k))<\lambda\left(\frac{1}{2}+\|\dot{f}\|_1\right)+\frac{\dot{\sigma}(s_k)^2}{2\lambda}.
\end{align}
In $\hat{t}$, $\sigma$ attains its maximum and since it is continuously differentiable on $(0,1)$ (this follows from (\ref{Thm121}) in combination with the fact, that $u$ is continuous on $(0,1)$ by part a) of Theorem \ref{Thm1.4}) it follows
\[
  \frac{\dot{\sigma}(s_k)^2}{2\lambda}\rightarrow 0\text{ for }k\rightarrow\infty
\]
and thereby
\begin{align}\label{contra}
\lim_{k\rightarrow\infty} F^*(\sigma(s_k))\leq \lambda\left(\frac{1}{2}+\|\dot{f}\|_1\right).
\end{align}
But the following calculation shows (see also figure \ref{fig2}), that the limit on the left-hand side is just the quantity $\omega_\infty$ from the assumptions of part c):
\begin{figure}[!ht]
 \centering
 \includegraphics[scale=1]{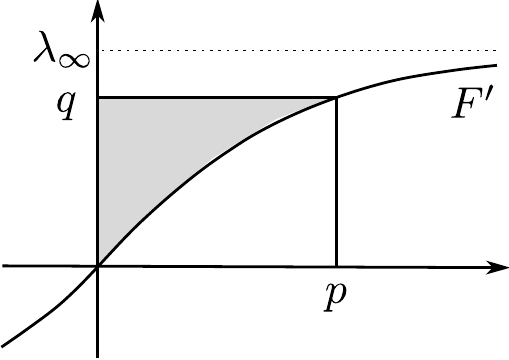}
 \caption{$\intop_{0}^q (F')^{-1}(t)\dt=pq-\intop_{0}^{p} F'(t)\dt$}\label{fig2}
\end{figure}

\begin{align*}
 \lim_{q\uparrow\lambda_\infty}F^*(q)=\lim_{q\uparrow\lambda_\infty}\intop_{0}^q (F^*)'(t)\dt=\lim_{q\uparrow\lambda_\infty}\intop_{0}^q (F')^{-1}(t)\dt
\end{align*}
\begin{align*}
 =\lim_{q\uparrow\lambda_\infty}q(F')^{-1}(q)-\intop_{0}^{(F')^{-1}(q)} F'(t)\dt\overset{p:=(F')^{-1}(q)}{=}\lim_{p\uparrow\infty}p F'(p)-F(p).
\end{align*}
Hence (\ref{contra}) is in contradiction to our requirements on $f$ and $\lambda$, thus our assumption $\mathrm{Sing}(u)\neq\emptyset$ is false. \qed
\end{section}

\begin{section}{Proof of Theorem \ref{Thm1.5}}
Let all the assumptions of Theorem \ref{Thm1.5} hold. In the following, we make the dependence of the minimizer on the parameter $\lambda$ more explicit by denoting with $u_\lambda$  the unique solution of problem (\ref{K}) for a given value $\lambda>0$.
Thanks to Theorems \ref{Thm1.2} and \ref{Thm1.4} we have the following properties:
\begin{enumerate}[(i)]
 \item $u_\lambda\in C^2\big([0,1]-\{\frac{1}{2}\}\big)$ (cf. Theorem 1.3 b)), $0\leq u_\lambda\leq 1$ a.e. and $u_\lambda$ satisfies
\begin{align*}
\left\{\begin{aligned}
\ddot{u}_\lambda=\lambda \frac{u_\lambda}{F''(\dot{u}_\lambda)},\,\dot{u}_\lambda(0)=0, &\text{ on }[0,\nicefrac{1}{2})\quad \mathrm{(1)},\\
\ddot{u}_\lambda=\lambda \frac{1-u_\lambda}{F''(\dul)},\,\dot{u}_\lambda(1)=0, &\text{ on }(\nicefrac{1}{2},1]\quad \mathrm{(2)},
\end{aligned}\right.
\end{align*}
\item \label{ii} $\ddot{u}_\lambda\geq 0$ on $[0,\frac{1}{2})$ and hence $\dot{u}_\lambda$ increases on $[0,\frac{1}{2})$; $\ddot{u}_\lambda\leq 0$ on $(\frac{1}{2},1]$ and hence $\dot{u}_\lambda$ decreases on $[0,\frac{1}{2})$,
\item $\dot{u}_\lambda\geq 0$ on $[0,\frac{1}{2})$ (due to $\dot{u}_\lambda(0)=0$ and (\ref{ii})) and hence $u_\lambda$ increases on $[0,\frac{1}{2})$.

\hspace{-1cm}\begin{minipage}{\textwidth}\vspace{0.3cm} Furthermore, we observe that the symmetry of our data $f$ with respect to the 
point $(\nicefrac{1}{2},\nicefrac{1}{2})$ is reproduced by $u_\lambda$:\end{minipage}
\item The two continuous branches of $u_\lambda$, $u_\lambda|_{[0,\frac{1}{2})}$ and  $u_\lambda|_{(\frac{1}{2},1]}$ are symmetric with respect to the point $(\frac{1}{2},\frac{1}{2})$, i.e.
\[
u_\lambda(t)=\underset{\mbox{$=:\tilde{u}_\lambda(t)$}}{\underbrace{1-u_\lambda(1-t)}},\, t\in[0,1]-\left\{\nicefrac{1}{2}\right\}
\]
\end{enumerate}
\textit{Proof of (iv)}. We show $K[\tilde{u}_\lambda]=K[u_\lambda]$. The result then follows from the uniqueness of the $K$-minimizer in $BV(0,1)$ (Theorem \ref{Thm1.2} a)). Let 
\[
h:=\lim_{t\downarrow \frac{1}{2}}u_\lambda(u)-\lim_{t\uparrow \frac{1}{2}}u_\lambda(u)
\] 
denote the height of the (possible) jump of $u_\lambda$ at $t=\nicefrac{1}{2}$. Then the distributional derivative of $u_\lambda$ is given by
\[
Du_\lambda=D^au_\lambda+h\delta_{1/2}
\]
and thus
\begin{align*}
K[u_\lambda]&=\intop_0^\frac{1}{2}\dot{u}_\lambda\dt+\intop_\frac{1}{2}^1\dot{u}_\lambda\dt+\lambda_\infty|h\delta_{1/2}|(0,1)+\frac{\lambda}{2}\intop_0^1(u_\lambda-f)^2\dt\\
&=\intop_0^\frac{1}{2}\dot{u}_\lambda\dt+\intop_\frac{1}{2}^1\dot{u}_\lambda\dt+\frac{|h|}{\mu-1}+\frac{\lambda}{2}\intop_0^1(u_\lambda-f)^2\dt.
\end{align*}
For $\tilde{u}_\lambda$ we obtain
\begin{align*}
&K[\tilde{u}_\lambda]\\
&=\intop_0^\frac{1}{2}\dot{u}_\lambda(1-t)\dt+\intop_\frac{1}{2}^1\dot{u}_\lambda(1-t)\dt+\lambda_\infty|h\delta_{1/2}|(0,1)+\frac{\lambda}{2}\intop_0^1(\tilde{u}_\lambda-f)^2\dt\\
&=\intop_0^\frac{1}{2}\dot{u}_\lambda\dt+\intop_\frac{1}{2}^1\dot{u}_\lambda\dt+\frac{|h|}{\mu-1}+\frac{\lambda}{2}\intop_0^1(\tilde{u}_\lambda-f)^2\dt,
\end{align*}
but clearly $\intop_0^1(\tilde{u}_\lambda-f)^2\dt=\intop_0^1(u_\lambda-f)^2\dt$ and hence $K[\tilde{u}_\lambda]=K[u_\lambda]$.\qed

Finally, we note that the value of $u_\lambda(0)$ tends to zero as $\lambda\rightarrow\infty$:
\begin{enumerate}[(v)]
 \item  \hspace{5cm}$\lim_{\lambda\rightarrow\infty}u_\lambda(0)=0.$
\end{enumerate}
\textit{Proof of (v)}. Since $u_\lambda$ is $K$-minimal in $BV(0,1)$ and $f\in BV(0,1)$, it must hold
\[
 K[u_\lambda]\leq K[f]=\lambda_\infty|\delta_{1/2}|(0,1)=\lambda_\infty=\frac{1}{\mu-1},
\]
and thus, due to properties (iii) and (iv)
\[
 \frac{\lambda}{2}u_\lambda(0)^2=2\frac{\lambda}{2}\intop_0^\frac{1}{2}u_\lambda(0)^2\dt\leq\frac{\lambda}{2}\intop_0^{\frac{1}{2}}(u_\lambda-f)^2\dt\leq K[u_\lambda]\leq K[f]=\frac{1}{\mu-1},
\]
so that 
\begin{align}\label{u(0)}
u_\lambda(0)\leq \sqrt{\frac{2}{\lambda(\mu-1)}}\xrightarrow{\lambda\rightarrow\infty}0. 
\end{align}
\qed

By property (iv), the continuity of $u_\lambda$ necessarily implies $u_\lambda(1/2)=1/2$. We can exploit this fact to prove that the minimizer develops jumps once we can show, that starting from a certain value of the parameter $\lambda$, $u_\lambda$ is bounded away from $1/2$ on $[0,1/2)$. To this end we make use of equation (1) from property (i):
\begin{align*}
 &\;\;\quad\ddot{u}_\lambda(t)=\lambda \frac{u_\lambda(t)}{F''(\dul(t))}\\
&\Leftrightarrow F''(\dul(t))\ddot{u}_\lambda(t)=\lambda u_\lambda(t)\\
&\Leftrightarrow \frac{d}{dt}F'(\dul(t))\dul(t)=\lambda u_\lambda(t)\dot{u}_\lambda(t).
\end{align*}
Integrating the latter equation from $0$ to $s$ for some $s\in[0,\frac{1}{2})$ yields
\begin{align*}
 &\;\;\quad\intop_0^s \frac{d}{dt}F'(\dul(t))\dul(t)\dt=\intop_0^s\lambda u_\lambda(t)\dot{u}_\lambda(t)\dt\\
&\Leftrightarrow \bigg[F'(\dul(t))\dul(t)\bigg]_0^s-\intop_0^s\underset{\mbox{$=\frac{d}{dt}F'(\dul(t))$}}{\underbrace{F'(\dul(t))\ddot{u}_\lambda(t)}}\dt=\left[\frac{\lambda}{2}u_\lambda(t)^2\right]_0^s,
\end{align*}
and with $\dot{u}_\lambda(0)=0$ and $F'(0)=0$ we arrive at
\begin{align}\label{conlaw}
 \dul(s)F'(\dul(s))-F(\dul(s))=\frac{\lambda}{2}\big(u_\lambda(s)^2-u_\lambda(0)^2\big).
\end{align}
Note that (\ref{conlaw}) formally corresponds to a law of conservation if we interpret eq. (1) as the equation of motion of a particle of mass $1/\lambda$ under the influence of a time-independent exterior force.

The left-hand side of (\ref{conlaw}) is nonnegative due to the convexity of $F$ and we therefore get:
\begin{align}\label{fint}
 u_\lambda(s)=\sqrt{u_\lambda(0)^2+\frac{2}{\lambda}\big(\dul(s)F'(\dul(s))-F(\dul(s))\big)}\quad\text{ for }s\in\left[0,\nicefrac{1}{2}\right).
\end{align}
From (\ref{fint}) we see, that if the left-hand side of (\ref{conlaw}) is bounded, then due to property (iv) $\ul$ is bounded below $1/2$ if we choose $\lambda$ large enough. But for our density $F$ from (\ref{F=Phi}) it holds (see Remark \ref{remthm13} (iii))
\begin{align*}
 \lim_{p\rightarrow\infty} p F'(p)-F(p)=
\begin{cases}
\infty,&\text{ for }1< \mu\leq 2,\\
\frac{1}{(\mu-1)(\mu-2)}, &\text{ for }2<\mu, 
\end{cases}
\end{align*}
and for $\mu>2$, the latter equation together with (\ref{u(0)}) and (\ref{fint}) gives
\begin{align*}
 u_\lambda(s)\leq\sqrt{\frac{2}{\lambda(\mu-1)}+\frac{2}{\lambda(\mu-1)(\mu-2)}}=\sqrt{\frac{2}{\lambda(\mu-2)}}\text{ for }s\in [0,\nicefrac{1}{2})
\end{align*}
which implies $\sup \limits_{0\leq s<\nicefrac{1}{2}}u_\lambda(s)<\nicefrac{1}{2}$, if $\lambda$ satisfies (\ref{lambdagrenz}). The corresponding lower bound on the infimum follows by the symmetry property (iv). \qed

\noindent\textit{Proof of Corollary \ref{Cor1.2} b)}. We define the critical value of $\lambda$ by
\[
 \lambda_{\mathrm{crit}}:=\sup\{\lambda\,:\,u_\lambda\text{ is continuous }\}.
\]
First we note that any minimizer $u_\lambda$ (independently of $\lambda$) satisfies $0\leq u_\lambda\leq \nicefrac{1}{2}$ on $[0,\nicefrac{1}{2})$ since otherwise ``cutting-off'' at height $\nicefrac{1}{2}$ would yield a $BV$-function for which the functional $K$ has a strictly smaller value. Thus, (\ref{fint}) implies
\begin{align*}
\frac{1}{2}\geq\sqrt{u_\lambda(0)^2+\frac{2}{\lambda}\big(\dul(s)F'(\dul(s))-F(\dul(s))\big)}\quad\text{ for }s\in\left[0,\nicefrac{1}{2}\right).
\end{align*}
Passing to the limit $s\rightarrow\nicefrac{1}{2}$ consequently gives (remember $\dot{u}_\lambda(s)\rightarrow\infty$ as $s\rightarrow\nicefrac{1}{2}$ since $\lambda>\lambda_\mathrm{crit}$)
\[
 \frac{1}{2}\geq \sqrt{u_\lambda(0)^2+\frac{2}{\lambda}\omega_\infty}
\]
which implies
\[
u_\lambda(0)^2 \leq \frac{\lambda-8\omega_\infty}{4\lambda}
\]
and consequently $\lambda\geq 8\omega_\infty$. The upper bound on $\lambda_\mathrm{crit}$ follows just like in the proof of Theorem \ref{Thm1.5} from the estimate (\ref{u(0)}) (with general $\lambda_\infty$ in place of $1/(\mu-1)$) and (\ref{fint}). \qed
\end{section}

\begin{section}{Proof of Theorem \ref{Thm1.6}}
Under the assumptions of Theorem \ref{Thm1.6} we let 
\[
 \mathrm{Reg}(u):=\big\{t\in[0,1]\,:\, u\text{ is }C^{1,1}\text{ on a neighbourhood of }t\big\}.
\]
From Theorem \ref{Thm1.2} b) we deduce that $\mathrm{Sing}(u):=[0,1]-\mathrm{Reg}(u)$ is a compact subset of $(0,1)$. Assume that $\mathrm{Sing}(u)\neq \emptyset$ and let $s$ denote the first singular point, thus $u\in C^{1,1}\big([0,s)\big)$ and therefore it holds
\begin{align}\label{six}
 \ddot{u}F''(\dot{u})=\lambda (u-f)\text{ a.e. on }(0,s).
\end{align}
From (\ref{six}) we deduce (compare the derivation of (\ref{conlaw})) the validity of
\begin{align}\label{conlaw2}
 \dot{u}(t)F'(\dot{u}(t))-F(\dot{u}(t))=\frac{\lambda}{2}\big(u(t)^2-u(0)^2\big)-\intop_{0}^tf(\tau)\dot{u}(\tau)\,d\tau\text{ for }t\in [0,s).
\end{align}
Clearly (\ref{conlaw2}) implies ($\omega(p):=p F'(p)-F(p)$)
\begin{align}\label{eight}
 |\omega(\dot{u}(t))|\leq\frac{\lambda}{2}+|Du|(0,1),\;t\in[0,s),
\end{align}
on account of $0\leq u,f\leq 1$ a.e. on $(0,1)$. By the convexity of $F$ (together with $F(0)=0$) we see that $\omega\geq 0$, $\omega(0)=0$, moreover
\begin{align*}
 \omega(p)=\intop_0^p\omega'(q)\,dq=\intop_0^qqF''(q)\,dq,
\end{align*}
thus $\omega$ is increasing with 
\begin{align}\label{nine}
 \lim_{p\rightarrow\infty}\omega(p)=\infty,\quad \lim_{p\rightarrow-\infty}\omega(p)=\infty
\end{align}
which follows from (\ref{muell}) together with assumption (\ref{muin}). Since we assume that $s$ is the first singular point of $u$, it must hold
\[
 \lim_{k\rightarrow\infty}|\dot{u}(t_k)|=\infty
\]
for a suitable sequence $t_k\uparrow s$, since otherwise $|\sigma(s)|<\lambda_\infty$ and hence $s\in\mathrm{Reg}(u)$ (cf. Remark (\ref{singularsigma})). This contradicts (\ref{eight}) on account of (\ref{nine}). \qed

\noindent\textit{Proof of Corollary \ref{Cor1.2} a)}. $\mathrm{Sing}(u)\neq \emptyset$ follows exactly along the same lines since now we have (\ref{nine}) due to our assumption $\omega_\infty=\infty$.\qed
\end{section}

\begin{section}{Proof of Theorem \ref{Thm1.3}}\label{section3}
We essentially have to show that for $\lambda<\lambda_\mu$ the conditions of Theorem 6, p. 295, in \cite{Th2} are fulfilled. Without further explanation we will adopt the notation of this work. First of all, we notice that due to our restriction $0\leq f(t)\leq 1$ we have that $\alpha(t)\equiv 0$ and $\beta(t)\equiv 1$ is a trivial lower and upper solution of (\ref{ode}), respectively, since
\begin{align*}
 0\geq \lambda\frac{0-f}{F''(0)}\quad\text{ as well as }\; 0\leq \lambda\frac{1-f}{F''(0)}
\end{align*}
due to $\|f\|_\infty\leq 1$ and $F''> 0$. Secondly, the right hand side of the equation (\ref{ode}) can be rewritten as
\[
 \Phi(t,v,\dot{v})=\lambda\frac{v-f(t)}{F''(\dot{v})}
\]
where $\Phi(t,y,p):=\lambda\frac{y-f(t)}{F''(p)}$ is a Carathéodory function if $f$ is merely measurable. Moreover, by (\ref{muell}) we can estimate $\Phi$ by
\begin{align*}
 |\Phi(t,y,p)|\leq \frac{\lambda}{c_1} \big(1+|p|\big)^\mu 
\end{align*}
and hence, letting $h(p):=\frac{\lambda}{c_1}(1+|p|)^\mu$, $\overline{h}(p)\equiv 1$, $r(t):=\eps$  for some $\eps>0$ and choosing $\lambda$ in such a way that
\begin{align}\label{lammu}
\lambda<\frac{1}{c_1(1+K\eps)}\intop_{1}^\infty \frac{s\,ds}{(1+s)^\mu}, 
\end{align}
where $K$ denotes the quantity $\sup\{s/h(s)\,|\,s\in[1,\infty]\}$, we find that  for some $L>0$ large enough,  $\Phi$ satisfies the following \textit{Bernstein-Nagumo-Zwirner} condition (compare \cite{Th2}, Definition 4):
\begin{align*}
\left\{\begin{aligned}
 &|\Phi(t,y,p)|\leq h(|p|)\overline{h}(p)+r(t)\text{ for all }(t,y)\in[0,1]\times[0,1]\text{ and }\\
 &\intop_{1}^L\frac{s\,ds}{h(s)}>1+K\eps.
\end{aligned}\right.
\end{align*}
The boundary conditions are formulated as set conditions, i.e. $(v(0),\dot{v}(0))\in \mathcal{J}(0)$ and $(v(1),\dot{v}(1))\in \mathcal{J}(1)$ for some closed connected subsets $\mathcal{J}(0),\mathcal{J}(1)\subset[0,1]\times\R$. In our case, we can choose
\[
 \mathcal{J}(0)=\mathcal{J}(1)=[0,1]\times\{0\}
\]
which corresponds to our Neumann condition. The verification, that the sets $\mathcal{J}(0)=\mathcal{J}(1):=[0,1]\times\{0\}$ are of ``compatible type 1`` in the sense of Definition 14 in \cite{Th2} is straightforward. Let us further define the sets $S_0,S_1,S_2$ and $S_3$  according to Definition 15 in \cite{Th2} (see figure \ref{fig3} below). Then we have
\[
 \mathcal{J}(0)\cap\{S_0\cup S_2\}=\mathcal{J}(1)\cap\{S_1\cup S_3\}=\{(0,0),(0,1)\}\neq\emptyset.
\]
\begin{figure}[h]
\centering
 \includegraphics[scale=0.9]{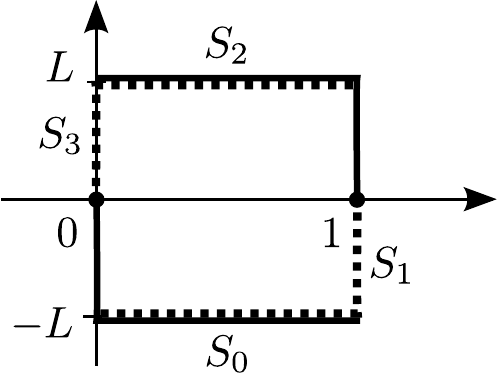}
 \caption{The sets $S_0,S_1,S_2$ and $S_3$.}\label{fig3}
\end{figure}
That is, all conditions of Theorem 4 are fulfilled and there is a solution $v\in W^{2,1}(0,1)$ of (\ref{ode}) with $0\leq v(t)\leq 1$ for almost all $t\in[0,1]$.
Note, that letting $\eps$ tend to zero in (\ref{lammu}) gives the postulated bound $\lambda_{\mu}$ for $\lambda$.

Let now $v\in W^{2,1}(0,1)$  be a solution of (\ref{ode}). We want to show, that $v$ coincides with the $K$-minimizer $u$ from Theorem \ref{Thm1.2}. By the convexity of the functional $J$, we see that for any $w\in C^{1,1}\big([0,1]\big)$ it holds 
\begin{align*}
 J[w]\geq J[v]+\langle DJ[v],w-v\rangle
\end{align*}
with
\[
 \langle DJ[v],w-v\rangle=\int\limits_{0}^{1}{F'(\dot{v})(\dot{w}-\dot{v})\dt}+\lambda\int\limits_{0}^{1}{(v-f)(w-v)\dt}.
\]
On account of $F'(0)=0$ we have
\begin{align*}
\int\limits_{0}^{1}{F'(\dot{v})(\dot{w}-\dot{v})\dt}&=\int\limits_{0}^{1}{\frac{d}{\dt}\big[F'(\dot{v})(w-v)\big]\dt}-\int\limits_{0}^{1}{F''(\dot{v})\ddot{v}(w-v)\dt}\\
 &=-\int\limits_{0}^{1}{F''(\dot{v})\ddot{v}(w-v)\dt}.
\end{align*}
By assumption, $v$ solves (\ref{ode}) a.e. on $(0,1)$, which implies
\begin{align}
\label{Thm17}  \langle DJ[v],w-v\rangle=\int\limits_{0}^{1}{(w-v)\big[F''(\dot{v})\ddot{v}-\lambda(v-f)\big]\dt}=0
\end{align}
for all $w\in C^{1,1}\big([0,1]\big)$. Thus, we get $J[v]\leq J[w]$ for all $w\in C^{1,1}\big([0,1]\big)$.
Now let $u$ denote the minimizer of $K$ in $BV(0,1)$. We can construct a sequence $u_k\in C^\infty\big([0,1]\big)$ such that
\[
 |Du_k|(0,1)\xrightarrow{k\rightarrow\infty} |Du|(0,1),\quad u_k\rightarrow u\text{ in }L^1(0,1),
\]
as well as
\[
 \sqrt{1+|Du_k|^2}(0,1)=\xrightarrow{k\rightarrow\infty}  \sqrt{1+| Du|^2}(0,1).
\]
To see this, consider 
\begin{align}
\hat{u}:\R\rightarrow[0,1],\; \hat{u}(t):=
 \begin{cases}
  u(0),\, &t\leq 0\\
  u(t),\, &0\leq t\leq 1\\
  u(1),\, &t\geq 1.
 \end{cases},
\end{align}
Since $u$ is of class $C^{1,1}$ near $0$ and $1$, it follows $\hat{u}\in BV_{\mathrm{loc}}(\R)$ and 
\begin{align}\label{Du0}
|D\hat{u}|(\{0\})=|D\hat{u}|(\{1\})=0.
\end{align}
Let $\eta\in C^\infty_0(\R)$ be a cut-off function such that $\eta\equiv 1$ on $[0,1]$ and consider a symmetric mollifier $\rho_\eps$, supported on the closed ball with radius $\eps>0$ around $0$. By the properties of mollification, $\hat{u}_\eps:=\rho_{\eps}\ast(\eta\hat{u})$ converges to $\hat{u}$ in $L^1(0,1)$ as $\eps\downarrow 0$. Moreover, due to (\ref{Du0}) and Proposition 3.7, p.121 in \cite{AFP} it holds
\[
 |D(\hat{u}_\eps)|(0,1)\xrightarrow{\eps\downarrow 0} |D(\eta\hat{u})|(0,1)=|Du|(0,1)
\]
and by similar arguments also
\[
 \intop_0^1\sqrt{1+|D(\hat{u}_\eps)|^2}\dt=\left|\rho_{\eps}\ast(\mathcal{L}^1,D(\eta\hat{u}))^T\right|(0,1)\xrightarrow{\eps\downarrow 0}  \sqrt{1+| Du|^2}(0,1).
\]
Hence $u_k:=\rho_{1/k}\ast \eta\hat{u}$ for $k\in\N$ has the desired properties. From Proposition 2.3 in \cite{AG} it follows
\[
 J[v]\leq J[u_k]=K[u_k]\xrightarrow{k\rightarrow\infty}K[u],
\]
and since $u$ is $K$-minimal, we conclude
\[
 K[u]\leq K[v]=J[v]\leq K[u],
\]
which means $K[u]=K[v]$ and thus $u=v$ by uniqueness of the $K$-minimizer. \qed
\end{section}
\newpage
\begin{section}{Comparison with a numerical example}
In this short appendix we would like to compare the theoretical considerations from  above to a numerical example which has been computed with the free software Scilab \footnote{http://www.scilab.org/}. Besides giving a confirmation of our previous results, this is mainly intended to show that none of our given bounds on the parameter $\lambda$ is actually sharp. In fact, we seem to obtain smooth solutions for values of $\lambda$ larger than $\max\{\lambda_\infty,\omega_\infty\}$ and discontinuous minimizers can occur below the threshold $\frac{8}{\mu-2}$ which has been predicted by Corollary \ref{Cor1.2} b).  It is still an open problem to determine  exact bounds, which clearly should depend on both $F$ and $f$ as well.

We choose the data $f$ from (\ref{specialdata}), i.e. $f$ is constant on $ [0,\nicefrac{1}{2}]$ and $(\nicefrac{1}{2},1]$ with a single jump of height $1$ at $t=\nicefrac{1}{2}$ and the $\mu$-elliptic density  $F(p)=\Phi_3(|p|)$ (remember, that by Theorem \ref{Thm1.6} there will be no singular minimizers for $\mu\leq2$ which is the  justification for our choice $\mu=3$). Then our $K$-minimizer $u$ should be smooth for $\lambda<8\omega_\infty=4$. 
\begin{figure}[!h]
\includegraphics[scale=1.5]{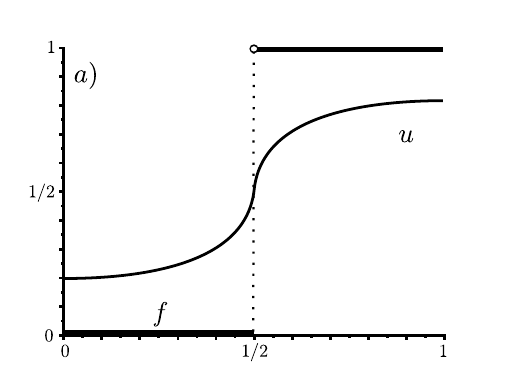}
\includegraphics[scale=1.5]{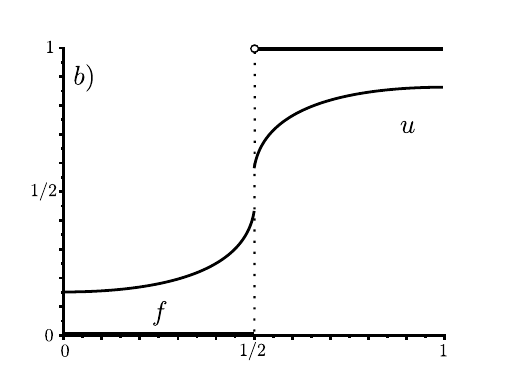}
\caption{Example plots of the $K$-minimizer $u$ for $\mu=3$ and $a)$ $\lambda=4$, $b)$ $\lambda=5$.}\label{fig4}
\end{figure}
In practice, we seem to get smooth solutions up to about $\lambda<4.16$. For $\lambda=4.16$ the tangent of $u$ at $t=\nicefrac{1}{2}$ becomes nearly vertical and for $\lambda>4.16$ the minimizer develops a jump. We have depicted in figure \ref{fig4} above exemplarily the graphs of $u$ for $\lambda=4$ and $\lambda=5$. Further we would like to note, that for $\lambda=4.16$ the value of $u(0)$ is approximately $0.183$ which yields for the bound (\ref{fint}) established in the proof of Theorem \ref{Thm1.5}
\begin{align*}
u(s)\leq\sqrt{0.183^2+\frac{1}{4.16}}\approx 0.523,
\end{align*}
and thus suits to our previous considerations quite well.
\end{section}

\noindent  Martin Fuchs, Jan M\"uller, Christian Tietz\\
Department of Mathematics\\
Saarland University\\
P.O.~Box 151150\\
66041 Saarbr\"ucken\\
Germany\\
fuchs@math.uni-sb.de, jmueller@math.uni-sb.de, tietz@math.uni-sb.de

\end{document}